\documentclass[a4paper, english, 11pt, twoside, reqno]{amsart}

\usepackage{babel}
\usepackage[T1]{fontenc}
\usepackage[cp1252]{inputenc}
\usepackage{amssymb}
\usepackage{amsmath, amsfonts, mathrsfs, amscd}
\usepackage{xcolor}
\usepackage{eucal}
\usepackage[all]{xy}
\usepackage{calligra}
\usepackage[all]{xypic}
\usepackage{graphicx}
\usepackage{hyperref}
\usepackage{amsthm}
\usepackage{pifont}
\usepackage{relsize}
\usepackage{mathtools}

\newtheorem*{theorem-intro}{Theorem}
\newtheorem{theorem}{Theorem}[section]
\newtheorem{lemma}[theorem]{Lemma}
\newtheorem{proposition}[theorem]{Proposition}
\newtheorem{corollary}[theorem]{Corollary}
\newtheorem{question}[theorem]{Question}

\theoremstyle{definition}

\newtheorem{remark}[theorem]{Remark}
\newtheorem{example}[theorem]{Example}

\newcommand{\Na}{\mathbb{N}}
\newcommand{\Z}{\mathbb{Z}}

\newcommand{\Co}{\mathbb{C}}
\newcommand{\F}{\mathbb{F}}

\newcommand{\B}{\mathcal{B}}

\newcommand{\M}{\mathrm{M}}

\newcommand{\PSL}{\mathbf{PSL}}
\newcommand{\SL}{\mathbf{SL}}
\newcommand{\GL}{\mathbf{GL}}
\newcommand{\Sp}{\mathbf{Sp}}

\newcommand\Hom{\operatorname{Hom}}
\newcommand\Rad{\operatorname{Rad}}
\newcommand\car{\operatorname{char}}

\newcommand*\sq{\mathbin{\vcenter{\hbox{\rule{0.6ex}{0.6ex}}}}}

\newcommand{\overbarM}[1]{\mkern 0mu\overline{\mkern -4.7mu#1\mkern-1.5mu}\mkern 0mu}
\newcommand{\overbarB}[1]{\mkern 0mu\overline{\mkern -1.5mu#1\mkern-1.5mu}\mkern 0mu}
\newcommand{\overbarK}[1]{\mkern 0mu\overline{\mkern -4.7mu#1\mkern-1.5mu}\mkern 1mu}

\allowdisplaybreaks

\begin{document}

\title[Integral Hopf orders in twists of group algebras]{Existence of integral Hopf orders in twists of group algebras}

\author[J. Cuadra and E. Meir]{Juan Cuadra and Ehud Meir}

\address{J. Cuadra: Universidad de Almer\'{\i}a, Dpto. Matem\'aticas, 04120 Almer\'{\i}a, Spain}
\email{jcdiaz@ual.es}

\address{E. Meir: Institute of Mathematics, University of Aberdeen, Fraser Noble Building, Aberdeen AB24 3UE, United Kingdom}
\email{meirehud@gmail.com}

\thanks{2020 \emph{Mathematics Subject Classification.} 16T05 (primary), 16H10 (secondary)}

\begin{abstract}
We find a group-theoretical condition under which a twist of a group algebra, in Movshev's way, admits an integral Hopf order. Let $K$ be a (large enough) number field with ring of integers $R$. Let $G$ be a finite group and $M$ an abelian subgroup of $G$ of central type. Consider\vspace{-0.75pt} the twist $J$ for $K\hspace{-0.8pt}G$ afforded by a non-degenerate $2$-cocycle on the character group $\widehat{M}$. We show\vspace{-0.75pt} that if there is a Lagrangian decomposition $\widehat{M} \simeq L \times \widehat{L}$ such that $L$ is contained in a normal abelian subgroup $N$ of $G$, then the twisted group algebra $(K\hspace{-0.8pt}G)_J$ admits a Hopf order $X$ over $R$. The Hopf order $X$ is constructed as the $R$-submodule generated by the primitive idempotents of $K\hspace{-1.1pt}N$ and the elements of $G$. It is indeed a Hopf order of $K\hspace{-0.8pt}G$ such that $J^{\pm 1} \in X \otimes_R X$. Furthermore, we give some criteria for this Hopf order to be unique. We illustrate this construction with several families of examples. As an application, we provide a further example of a simple and semisimple complex Hopf algebra that does not admit integral Hopf orders.
\end{abstract}
\maketitle

\section{Introduction}

Orders have played an important role in algebra and number theory for a long time. To contextualize our work, we begin with a brief overview of their development within noncommutative algebra and their connection to Hopf algebras.

\subsection{Noncommutative arithmetic}
In his review \cite{Gu} of Reiner’s monograph \cite{Re}, Gustafson traces the origins of the theory of orders, from the ring of integers in a number field to its noncommutative counterpart. In 1916, Brandt introduced orders in generalized quaternion algebras to study quadratic forms. Soon after, in 1919, Hurwitz published his treatise on the now well-known integral quaternions, \cite[Chapter 11]{Vo}. In 1932, Hasse’s work on maximal orders in division algebras marked a milestone in class field theory: the computation\textemdash{}jointly with Brauer and Noether\textemdash{}of the Brauer group of local and global fields and, ultimately, that of a number field (see \cite{Ro}). Reiner refers to the theory of maximal orders as \emph{''noncommutative arithmetic''}. As outlined in \cite{Gu}, this theory finds applications in topology, geometry, algebraic number theory, ring and module theory, and representation theory. \par \vspace{2pt}

In the representation theory of finite groups, several results illustrate how the arithmetical properties of an order help derive structural features of the algebra containing it. For a finite group $G$, a prominent example of an order is the group ring $\Z G$ within the group algebra $\Co G$. Specifically, $\Z G$ is a subring of $\Co G$, finitely generated as a $\Z$-module, and $\Z G$ spans $\Co G$ as a vector space over $\Co$. These conditions ensure that every element of $\Z G$ is integral, meaning it satisfies a monic polynomial with coefficients in $\Z$. This integrality underpins the proof of Frobenius' Theorem, which states that the degree of every complex irreducible representation of $G$ divides the order of $G$ (see \cite[Proposition 9.32]{CR}). A refinement of this result, due to It\^{o}, asserts that if $N$ is a normal abelian subgroup of $G$, then the degree of every complex irreducible representation of $G$ divides the index of $N$ in $G$.

\subsection{Hopf orders in semisimple Hopf algebras}
Semisimple Hopf algebras, and more broadly, tensor categories, establish a natural framework where the representation theory of finite groups fits and inspires a wealth of deep mathematical problems. Among these, Kaplansky's sixth conjecture (1975) remains open. Frobenius' Theorem gives a key divisibility property for the degrees of irreducible representations of a finite group. Kaplansky's conjecture aims to extend this property to semisimple Hopf algebras (see \cite[Section 5]{So} and \cite[Subsection 4.2]{A}). Concretely, it asserts that, for any complex semisimple Hopf algebra $A$, the dimension of every irreducible representation of $A$ divides the dimension of $A$. \par \vspace{2pt}

When $\Co G$ is endowed with its canonical Hopf algebra structure, $\Z G$ becomes a Hopf order of $\Co G$. First studied by Larson in \cite{L1}, the concept of Hopf order brings a number-theoretic perspective to the study of Hopf algebras. Alongside other tools, it provides a bridge between characteristic $0$ and characteristic $p$, enabling the use of the local-global principle. Applications of this principle appear in \cite[Proposition 20.1 and p. 108]{Ch2}, in the cocommutative setting, and in \cite{EG3}, \cite{CEW1} and \cite{CEW2}, in the quantum group setting. \par \vspace{2pt}

Let $H$ be a finite-dimensional Hopf algebra over a field $K$, and let $R$ be a subring of $K$. A Hopf order of $H$ over $R$ is a Hopf algebra $X$ over $R$, finitely generated and projective as an $R$-module, such that the canonical map $X \otimes_R K \rightarrow H$ is an isomorphism (see Subsection \ref{SHor} for details). Larson \cite[Proposition 4.2]{L1} proved that if a complex semisimple Hopf algebra admits a Hopf order over a number ring, it satisfies Kaplansky's sixth conjecture. This result revisits the core argument of Frobenius' Theorem and raises the question of whether every complex semisimple Hopf algebra admits a Hopf order over a number ring (integral Hopf order). \par \vspace{2pt}

For cocommutative Hopf algebras (group algebras), a well-developed theory of Hopf orders emerged from Larson's seminal papers \cite{L1} and \cite{L2}; see, for example, \cite{Ch2} and \cite{U} and the references therein. The theory of quantum groups emphasizes Hopf algebras that are neither commutative nor cocommutative. However, despite the extensive growth of this field, the study of Hopf orders has received little attention. \par \vspace{2pt}

In \cite{CM1}, we showed that, unlike group algebras, complex semisimple Hopf algebras may not admit integral Hopf orders. This phenomenon was further examined in \cite{CM3} and \cite{CCM}. The semisimple Hopf algebras studied in these works are twisted group algebras of non-abelian groups. Their coalgebra structure and antipode are modified following Movshev's construction, while their algebra structure remains unchanged (see Subsections \ref{Sc:drt} and \ref{Sc:mv} for details). As algebras, they are group algebras and thus satisfy the conjecture. Our results indicate that the approach used to prove Frobenius' Theorem is insufficient to establish Kaplansky's sixth conjecture. Moreover, all the examples analyzed are simple Hopf algebras. This led us to ask in \cite{CM3} whether simplicity is the reason behind the absence of integral Hopf orders. \par \vspace{2pt}

The twisting procedure has a categorical interpretation. In fact, it preserves the representation category of $H$: if $J$ is a twist for $H$, then $\text{Rep}(H)$, is isomorphic, as a tensor category, to $\text{Rep}(H_J)$, the category of representations of the twisted Hopf algebra $H_J$. There is a one-to-one correspondence between the equivalence classes of fiber functors $\text{Rep}(H) \to \text{Vec}_K$ and the equivalence classes of twists of $H$ (see \cite[Section 5.14]{EGNO}). The formulation of Frobenius' Theorem in the framework of tensor categories, together with the existing results in this context (see \cite[Theorem 1.5]{ENO1} and \cite[Section 5]{ENO2}), confirms that Kaplansky's sixth conjecture is a difficult problem. \par \vspace{2pt}

Movshev's method for twisting a group algebra has yielded structural results in Hopf algebra theory. For example, it appears in the classification of triangular and cotriangular Hopf algebras (\cite{AEGN}, \cite{EG4}, and \cite{EOV}), and in the construction of simple and semisimple ones (\cite{N} and \cite{GN}). Against this background, it is natural to ask when a twist of a group algebra admits an integral Hopf order.

\subsection{Subject of this paper}
This paper addresses that question. We shift our focus from the non-existence of integral Hopf orders in twisted group algebras to identifying conditions under which  they do exist, how to construct them, and when they are unique. Our motivation comes from the example of a Hopf order that we introduced in \cite[Proposition 4.1]{CM3} for a twisted group algebra on the symmetric group $S_4$. Most of this paper is devoted to fitting this construction into a general group-theoretical framework.

\subsection{Results and method of proofs}
Let $K$ be a number field with ring of integers $R$. Let $G$ be a finite group and $M$ an abelian\vspace{-0.5pt} subgroup of $G$. Suppose that $K$ is large enough so that $K\hspace{-0.9pt}M$ splits as an algebra. Let $\widehat{M}$ be the character\vspace{1.25pt} group of $M$. Take the set $\{e_{\phi}\}_{\phi \in \widehat{M}}$ of orthogonal primitive idempotents in $K\hspace{-0.9pt}M$ giving the Wedderburn\vspace{-3pt} decomposition.\vspace{-0.6pt} Recall that $e_{\phi} := \frac{1}{\vert M \vert} \sum_{m \in M} \phi(m^{-1}) m.$ If $\omega: \widehat{M} \times \widehat{M} \rightarrow K^{\times}$ is a (normalized) 2-cocycle, then
$$J:=\sum_{\phi,\psi \in \widehat{M}} \omega(\phi,\psi) e_{\phi} \otimes e_{\psi}$$
is a twist for the group algebra $K\hspace{-0.9pt}G$ (see Subsections \ref{Sc:drt} and \ref{Sc:mv}). This algebra can be endowed with a new Hopf algebra structure, denoted by $(K\hspace{-0.9pt}G)_J$, as follows. The algebra structure remains unchanged. The coproduct, counit, and antipode are defined from those of $K\hspace{-0.9pt}G$ in the following way:
$$\Delta_{J}(g)= J\Delta(g) J^{-1}, \hspace{0.7cm} \varepsilon_J(g)=\varepsilon(g), \hspace{0.7cm} S_J(g)=U_J\hspace{1pt}S(g)\hspace{0.5pt}U_J^{-1}, \hspace{0.7cm} \forall g \in G.$$
Here, $U_J:= \sum_{\phi \in \widehat{M}} \hspace{1pt}\omega(\phi,\phi^{-1}) e_{\phi}.$ \par \vspace{2pt}

Our first result stems from the following observation. The cocycle $\omega$ is cohomologous to one with values in roots of unity. We can replace $J$ by a cohomologous twist and extend $K$, if necessary, to ensure that $\omega$ takes values in $R$. Suppose that $M$ is contained in a normal abelian subgroup $N$ of $G$. Extend $K$ again, if necessary, to ensure that $K\hspace{-0.9pt}N$ splits as an algebra. Since $K\hspace{-0.9pt}N$ is commutative, every idempotent of $K\hspace{-0.9pt}M$ is a sum of idempotents of $K\hspace{-0.9pt}N$. Take the full set $\{e^N_{\nu}\}_{\nu \in \widehat{N}}$ of orthogonal primitive idempotents in $K\hspace{-0.9pt}N$ (the superscript $N$ is placed to distinguish\vspace{-0.5pt} these idempotents in different group algebras). Since $N$ is normal,\vspace{1.5pt} $G$ acts on $\widehat{N}$ by $(g \triangleright \nu)(n)=\nu(g^{-1}ng)$ for all $n \in N$. The following rule holds\vspace{-1.5pt} in $K\hspace{-0.8pt} G$: $(e_{\nu}^Ng)(e_{\nu'}^Ng')=e_{\nu}^Ne^N_{g \hspace{0.65pt}\triangleright\hspace{0.65pt} \nu'}gg'$. Thus, the $R$-subalgebra $X$ of $K\hspace{-0.8pt} G$ generated by the set $\{e_{\nu}^Ng: \nu \in \widehat{N}, g\in G\}$ is finitely\vspace{1pt} generated as an $R$-module. One can see\vspace{1pt} that $X$ is a Hopf order of $K\hspace{-0.8pt} G$ by using the formulas:\linebreak $\Delta(e^N_{\nu})=\sum_{\eta \in \widehat{N}} e^N_{\eta} \otimes e^N_{\eta^{-1}\nu}, \hspace{3pt} \varepsilon(e^N_{\nu})=\delta_{\nu,1},$ and $S(e^N_{\nu})=e^N_{\nu^{-1}}$. The idempotents of $K\hspace{-0.9pt}M$ belong\vspace{1pt} to $X$. Hence, $J^{ \pm 1} \in X \otimes_R X$. This implies that $\Delta_J(X) \subseteq X \otimes_R X$, $\varepsilon_J(X) \subseteq R,$ and $S_J(X) \subseteq X$. So, $X$ is a Hopf order of $(K\hspace{-0.9pt}G)_J$ over $R$. \par \vspace{2pt}

The previous idea\vspace{-0.25pt} can be refined through the concept of Lagrangian decomposition. Consider the\vspace{-0.25pt} skew-symmetric pairing \hspace{-1pt}$\B_{\omega}:\widehat{M} \times \widehat{M} \rightarrow K^{\times}$ associated to $\omega$. It is defined as $\B_{\omega}(\phi,\psi) = \omega(\phi,\psi)\omega(\psi,\phi)^{-1}$ for all $\phi,\psi \in \widehat{M}.$ Assume\vspace{1pt} that $\B_{\omega}$ is non-degenerate. We alternatively say that $\omega$ is non-degenerate. It is known that $M$ admits a non-degenerate 2-cocycle if and only if $M$ is of central type. This means that $M \simeq E \times E$ for some abelian group $E$. \par \vspace{2pt}

A subgroup $L$ of $\widehat{M}$ is called Lagrangian if $L=L^{\perp}$, where $\perp$ is taken with respect to $\B_{\omega}$. A Lagrangian subgroup produces a short exact sequence $1 \rightarrow L \rightarrow \widehat{M} \rightarrow \widehat{L} \rightarrow 1.$ If it splits, then $\widehat{M} \simeq L\times \widehat{L}$. Such\vspace{-0.5pt} a decomposition is called a Lagrangian decomposition of $\widehat{M}$. A Lagrangian decomposition\vspace{-0.25pt} always exists and has the following property: writing\vspace{1pt} every element of $\widehat{M}$ as a pair $(l,\lambda)$, with $l \in L,\lambda \in \widehat{L}$, the cocycle $\omega$ is (up to coboundary) given\vspace{1pt} by $\omega((l,\lambda),(l',\lambda')) = \lambda(l')$. Thanks to this property, we can show in Lemma \ref{lemJ} that $J$ and $J^{-1}$ can be expressed as:
$$J^{\pm 1} \hspace{1pt}=\hspace{1pt} \sum_{\lambda \in \widehat{L}} e_{\lambda}^{L} \otimes \lambda^{\pm 1} \hspace{1pt}=\hspace{1pt} \sum_{l \in L} l^{\pm 1} \otimes e_l^{\widehat{L}}.$$
By applying $\widehat{\hspace{5pt}\cdot\hspace{5pt}}$ to $\widehat{M} \simeq L\times \widehat{L}$, we get that $M \simeq \widehat{L} \times  L$. This allows to view $L$ as a subgroup of $M$. \par \vspace{2pt}

Our first main result (Theorem \ref{thm:main1}) is stated as follows:

\begin{theorem-intro}
Let $K$ be a (large enough) number field with ring\vspace{1pt} of integers $R$. Let $G$ be a finite group and $M$ an abelian subgroup of $G$ of central type. Consider the twist \linebreak $J$ in $K\hspace{-1pt}M \otimes K\hspace{-1pt}M$ afforded by a non-degenerate $2$-cocycle $\omega: \widehat{M} \times \widehat{M} \rightarrow K^{\times}$. \par \vspace{2pt}

Fix a Lagrangian decomposition $\widehat{M} \simeq L \times \widehat{L}$. Suppose that $L$ (viewed as inside of $M$) is contained in a normal abelian subgroup $N$ of $G$. Then, $(K\hspace{-0.8pt}G)_J$ admits a Hopf order $X$ over $R$.
\end{theorem-intro}

As before, $X$ is the Hopf order of $K\hspace{-0.8pt} G$ generated by $\{e_{\nu}^Ng: \nu \in \widehat{N}, g\in G\}$ and it satisfies that $J^{ \pm 1} \in X \otimes_R X$. Under the extra hypothesis that the action of $G/N$ on $N$ induced by conjugation is faithful, $X$ can be characterized as the unique Hopf order of $(K\hspace{-0.69pt}G)_J$ containing all the primitive idempotents of $K\hspace{-0.9pt}N$ (Proposition \ref{cond:unique}). \par \vspace{2pt}

We wonder if, up to cohomologous twist, every Hopf order of a twisted group algebra $(K\hspace{-0.69pt} G)_J$ is a Hopf order $X$ of $K\hspace{-0.69pt} G$ such that $J^{\pm 1} \in X \otimes_R X$. All examples of integral Hopf orders in twisted group algebras known so far arise in this form. \par \vspace{2pt}

Our second main result (Theorem \ref{thm:main2}) deals with the uniqueness of the Hopf order constructed above in the case of semidirect products of groups:

\begin{theorem-intro}
Let $K$ be a (large enough) number field with ring of integers $R$. Consider the semidirect product $G := N \rtimes Q$ of two finite groups $N$ and $Q$, with $N$ abelian. Let $L$ and $P$ be abelian subgroups of $N$ and $Q$, respectively. Set $M=LP$. Let $\tau \in Q$. Suppose that $N,Q,L,P,$ and $\tau$ satisfy the following conditions:
\begin{enumerate}
\item[{\it (i)}] $L$ and $P$ are isomorphic and commute with one another. \vspace{2pt}
\item[{\it (ii)}] $Q$ acts on $N$ faithfully. \vspace{2pt}
\item[{\it (iii)}] $N=L \oplus (\tau \cdot L)$, where $N$ is written additively. \vspace{2pt}
\item[{\it (iv)}] $N^{\tau} \neq \{1\}$. \vspace{2pt}
\item[{\it (v)}] $\widehat{N}^{\sigma \tau}=\big(\widehat{N}^{\tau}\big) \cap \big(\widehat{N}^{\sigma \tau \sigma^{-1}}\big)=\{\varepsilon\}$ for every $\sigma \in P$ with $\sigma \neq 1$.
\end{enumerate}
Let $J$ be the twist in $K\hspace{-1pt}M \otimes K\hspace{-1pt}M$ arising from an isomorphism $f:P \rightarrow \widehat{L}$ (see Section \ref{ex-uniq}). Then, the Hopf order of $(K\hspace{-0.9pt}G)_J$ over $R$ generated by the primitive idempotents of $K\hspace{-1.1pt}N$ and the elements of $Q$ is unique.
\end{theorem-intro}

To establish the uniqueness, it is enough to prove that any Hopf order $Y$ of $(K\hspace{-0.9pt}G)_J$ over $R$ contains the idempotent $e^L_{\varepsilon}$ (Propositions \ref{cond:unique} and \ref{squeeze}). This is achieved by obtaining $e^L_{\varepsilon}$ from certain elements of $(K\hspace{-0.9pt}G)_J$ that must belong to  $Y$. Set, for brevity, $H=(K\hspace{-0.9pt}G)_J$. The dual Hopf order $Y^{\star}$ consists of those $\varphi \in H^*$ such that $\varphi(Y) \subseteq R$. Any character of $H$ belongs to $Y^{\star}$ and any cocharacter of $H$ belongs to $Y$ (Proposition \ref{character}). We can construct elements in $Y$ by manipulating characters and cocharacters and by using the operations of Hopf order of $Y$ and $Y^{\star}$ and the evaluation map $Y^{\star} \otimes_R Y \rightarrow R$. Another tool that helps to obtain elements in $Y$ is the following (Proposition \ref{sub}): if $A$ is a Hopf subalgebra of $H$, then $Y \cap A$ is a Hopf order of $A$. We will prove that $e^L_{\varepsilon} \in Y$ by exploiting these facts. \par \vspace{2pt}

The previous strategy is reinforced with the knowledge of the cocharacters of $(K\hspace{-0.9pt}G)_J$. In Proposition \ref{Prop:irr}, we determine the irreducible cocharacters of a general twisted group algebra $(K\hspace{-0.9pt}G)_J$. For any $\tau \in G$, we show in Proposition \ref{vipcharacter} that the element $|M|e^M_{\varepsilon} \tau e^M_{\varepsilon}$ is a cocharacter of $(K\hspace{-0.9pt}G)_J$ when $\omega$ is non-degenerate. \par \vspace{2pt}

Our second theorem is illustrated with a family of examples in which $N=\F_q^{2n},$ $Q=\SL_{2n}(q)$, and $Q$ acts on $N$ in the natural way (see Section \ref{example}). Here, $\F_q$ is the finite field with cardinality $q$. The role of $\SL_{2n}(q)$ can be also played by $\GL_{2n}(q)$ or $\Sp_{2n}(q)$. The subgroup $P$ of $Q$ is defined by means of an algebra homomorphism $\Phi:\F_{q^n}\to \M_n(\F_q)$ induced by the product of $\F_{q^n}$. \par \vspace{2pt}

The group $\F_q^{2n} \rtimes \SL_{2n}(q)$ embeds in $\PSL_{2n+1}(q)$. The twist $J$ for $K(\F_q^{2n} \rtimes \SL_{2n}(q))$ of the previous theorem is also a twist for $K \PSL_{2n+1}(q)$. As an application of our results, we prove in Theorem \ref{PSL2n1} that $(K \PSL_{2n+1}(q))_J$ does not admit a Hopf order over $R$. The Hopf algebra $(\Co \PSL_{2n+1}(q))_J$ provides a further example of simple and semisimple complex Hopf algebra that does not admit a Hopf order over any number ring. \par \vspace{2pt}

\subsection{Organization of the paper} \enlargethispage{\baselineskip}
The remainder of this paper is organized as follows: \par \vspace{2pt}

In Section 2, we recall some background material on Hopf orders, Drinfeld's twist, and Movshev's method of twisting a group algebra. In Section 3, we discuss the coalgebra structure of a twisted group algebra and describe its irreducible cocharacters. \par \vspace{2pt}

In Section 4, we establish our first main theorem, after a preliminary discussion on Lagrangian decompositions. We also characterize the Hopf order constructed here in several ways, and underline, through an example, that our method of construction can produce different Hopf orders. The problem of the uniqueness is tackled in Section 5, where we establish our second main theorem. The above-mentioned families of examples are presented in Section 6. \par \vspace{2pt}

Lastly, Section 7 deals with the non-existence of integral Hopf orders for a twist of the group algebra on $\PSL_{2n+1}(q)$.

\subsection{Acknowledgements}
The work of Juan Cuadra was partially supported by the Spanish Ministry of Science and Innovation, through the grant PID2020-113552GB-I00 (AEI/FEDER, UE), by the Andalusian Ministry of Economy and Knowledge, through the grant P20\underline{ }00770, and by the research group FQM0211. \par \smallskip

The authors would like to thank the referee for their comments and suggestions, which helped improve the motivation and broaden the scope of the introduction.

\section{Preliminaries}

The preliminary material necessary for this paper is the same as that of \cite{CM1}, \cite{CM3}, and \cite{CCM}. For convenience, we briefly collect here the indispensable content and refer the reader to there for further information.

\subsection{Conventions and notation}
We will work over a base field $K$ (mostly a number field). Unless otherwise specified, vector spaces, linear maps, and undecorated tensor products $\otimes$ will be over $K$. Throughout, $H$ will stand for a finite-dimensional Hopf algebra over $K$, with identity element $1_H$, coproduct $\Delta$, counit $\varepsilon$, and antipode $S$. The dual Hopf algebra of $H$ will be denoted by $H^*$. For the general aspects of  the theory of Hopf algebras, our reference books are \cite{Mo} and \cite{Ra}.

\subsection{Hopf orders}\label{SHor} Let $R$ be a subring of $K$ and $V$ a finite-dimensional vector space over $K$. A \emph{lattice of\hspace{1pt} $V$\hspace{-1.5pt} over $R$} is a finitely generated and projective $R$-submodule $X$ of $V$ such that the natural map $X \otimes_R K \rightarrow V$ is an isomorphism. Under this isomorphism, $X$ corresponds to the image of $X \otimes_R R$. \par \vspace{2pt}

A \textit{Hopf order of $H$ over $R$} is a lattice $X$ of $H$ which is closed under the Hopf algebra operations; that is, $1_H \in X$, $XX \subseteq X$, $\Delta(X)\subseteq X\otimes_{R} X$, $\varepsilon(X) \subseteq R,$ and $S(X)\subseteq X$. (For the condition on the coproduct, we use the natural identification of $X\otimes_{R} X$ as an $R$-submodule of $H\otimes H$.) Our reference books for the theory of Hopf orders in group algebras are \cite{Ch2} and \cite{U}. \par \vspace{2pt}

In the next three results, $K$ is assumed to be a number field with ring of integers $R$. Hopf orders are understood to be over $R$.

\begin{lemma}\cite[Lemma 1.1]{CM1} \label{dual}
Let $X$ be a Hopf order of $H$.
\begin{enumerate}
\item[{\it (i)}] The dual lattice $X^{\star}:=\{\varphi \in H^* : \varphi(X) \subseteq R\}$ is a Hopf order of $H^*$. \vspace{2pt}
\item[{\it (ii)}] The natural isomorphism $H \simeq H^{**}$ induces an isomorphism $X \simeq X^{\star \star}$ of Hopf orders. \vspace{2pt}
\end{enumerate}
\end{lemma}

The importance of the dual order for us ultimately lies in the following:

\begin{proposition}\cite[Proposition 1.2]{CM1} \label{character}
Let $X$ be a Hopf order of $H$. Then:
\begin{enumerate}
\item[{\it (i)}] Every character of $H$ belongs to $X^{\star}$. \vspace{2pt}
\item[{\it (ii)}] Every character of $H^*$ (cocharacter of $H$) belongs to $X$. In particular, $X$ contains every group-like element of $H$.
\end{enumerate}
\end{proposition}

We will often use the following technical tool:

\begin{proposition}\cite[Proposition 1.9]{CM1} \label{sub}
Let $X$ be a Hopf order of $H$. If $A$ is a Hopf subalgebra of $H$, then $X\cap A$ is a Hopf order of $A$.
\end{proposition}

\subsection{Drinfeld twist} \label{Sc:drt}
An invertible element $J:=\sum J^{(1)} \otimes J^{(2)}$ in $H \otimes H$ is called a {\it twist} for $H$ provided that:
$$\begin{array}{c}
(1_H \otimes J)(id \otimes \Delta)(J)=(J \otimes 1_H)(\Delta \otimes id)(J), \quad \textrm{and} \vspace{4pt} \\
(\varepsilon \otimes id)(J)=(id  \otimes \varepsilon)(J)=1_H.
\end{array}$$
The {\it Drinfeld twist} of $H$ is the new Hopf algebra $H_J$ constructed as follows: $H_J=H$ as an algebra, the counit is that of $H$, and the coproduct and antipode are defined by:
$$\Delta_J(h)=J \Delta(h)J^{-1} \qquad \textrm{and} \qquad S_J(h)=U_J\hspace{1pt}S(h)\hspace{0.5pt}U_J^{-1} \qquad \forall h \in H.$$
Here, $U_J:=\sum J^{(1)}S(J^{(2)})$. Writing $J^{-1}=\sum J^{-(1)} \otimes J^{-(2)}$, we have that $U_J^{-1}=\sum S(J^{-(1)})J^{-(2)}.$ \par \vspace{2pt}

We stress that if $A$ is a Hopf subalgebra of $H$ and $J$ is a twist for $A$, then $J$ is a twist for $H$. \par \vspace{2pt}

Our main results will rely on the following fact, which is easy to prove:

\begin{proposition}\cite[Proposition 2.4]{CM3} \label{twistorder}
Let $H$ be a Hopf algebra over $K$ and $J$ a twist for $H$. Let $R$ be a subring of $K$ and $X$ a Hopf order of $H$ over $R$. Assume that $J$ and $J^{-1}$ belong to $X \otimes_R X$. Then, $X$ is a Hopf order of $H_J$ over $R$.
\end{proposition}

In a similar fashion, we denote by $X_J$ the Drinfeld twist of $X$.

\subsection{Construction of twists for group algebras} \label{Sc:mv}
Movshev devised in \cite{Mv} the following ingenious method of constructing twists for a group algebra. Let $M$ be a finite abelian group. Assume\vspace{-1pt} that $\car K \nmid \vert M \vert$ and that $K$ is large enough for the group algebra $K\hspace{-0.9pt}M$ to split. Consider\vspace{1pt} the character group $\widehat{M}$ of $M$. For $\phi \in  \widehat{M}$, the primitive idempotent of $K\hspace{-0.9pt}M$ corresponding to $\phi$ is
$$e_{\phi} := \frac{1}{\vert M \vert } \sum_{m \in M} \phi(m^{-1}) m.$$
If $\omega: \widehat{M} \times \widehat{M} \rightarrow K^{\times}$ is a normalized 2-cocycle, then
$$J_{M, \omega}:=\sum_{\phi,\psi \in \widehat{M}} \omega(\phi,\psi) e_{\phi} \otimes e_{\psi}$$
is a twist for $K\hspace{-0.9pt}M$. \par \vspace{2pt}

Suppose now that $G$ is a finite group and $M$ is an abelian subgroup of $G$. Then, $K\hspace{-0.9pt}M$ is a Hopf subalgebra of $K\hspace{-0.9pt}G$ and, consequently, $J_{M, \omega}$ is a twist for $K\hspace{-0.9pt}G$. (It is pertinent to mention that not all twists of $K\hspace{-0.9pt}G$ arise in this way, see \cite{EG2}.) In the sequel, we will omit the subscripts in the twist and simply write $J$.

\section{A distinguished cocharacter}
\enlargethispage{\baselineskip}

Let $\tau \in G$. The aim of this section is to prove that the element $|M|e_{\varepsilon} \tau e_{\varepsilon}$ is a cocharacter of $(K\hspace{-0.9pt}G)_J$ and, hence, it is contained in every Hopf order of $(K\hspace{-0.9pt}G)_J$. To achieve this, some preparations are necessary.

\subsection{Irreducible representations of twisted group algebras over abelian groups}
Let $\hspace{3pt}\overbarK{K}\hspace{0pt}$ denote the algebraic closure of the number field $K$. Consider the following two abelian groups (see \cite[Section 2.1, p. 31]{Ka} and \cite[Section 1.2, p. 18]{Ka2} for the definitions):
\begin{enumerate}
\item[$\sq$] $H^2(M,\hspace{3pt}\overbarK{K}^{\times})$, the second cohomology group of $M$ with values in $\hspace{3pt}\overbarK{K}^{\times}\hspace{0pt}$. \vspace{2pt}
\item[$\sq$] $P_{sk}(M,\hspace{3pt}\overbarK{K}^{\times})$, the group of skew-symmetric pairings of $M$ with values in $\hspace{3pt}\overbarK{K}^{\times}\hspace{0pt}$.
\end{enumerate}
Every $2$-cocycle $\omega$ gives rise to a skew-symmetric pairing $\B_{\omega}$ defined by
$$\B_{\omega}(m,m') = \omega(m,m')\omega(m',m)^{-1} \hspace{8pt} \forall m,m' \in M,$$
and $\B_{\omega}$ depends only on the cohomology class of $\omega$. We know from \cite[Lemma 2.2, p. 19, and Theorem 3.6, p. 31]{Ka2} that the map
$$\B:H^2(M,\hspace{3pt}\overbarK{K}^{\times}) \rightarrow P_{sk}(M,\hspace{3pt}\overbarK{K}^{\times}), [\omega] \mapsto \B_{\omega},$$
is an isomorphism of abelian groups. \par \vspace{2pt}

The \emph{radical} of $\omega$ is defined to be the radical of the pairing $\B_{\omega}$. That is, $$\Rad(\omega)=\{m \in M : \omega(m,m')=\omega(m',m) \hspace{5pt} \forall m' \in M\}.$$
Clearly, $\Rad(\omega)$ is a subgroup of $M$. \par \vspace{2pt}

Recall that $\omega$ is said to be \emph{non-degenerate} if the pairing $\B_{\omega}$ is so; equivalently, if $\Rad(\omega)=\{1\}$. Being non-degenerate is a property preserved under multiplication by coboundaries. Let $\pi:M \rightarrow M/\Rad(\omega)$ denote the canonical projection. The pairing $\mathcal{B}_{\omega}$ on $M$ induces a skew-symmetric pairing $\overbarB{\mathcal{B}}_{\omega}$ on $M/\Rad(\omega)$ such that $\mathcal{B}_{\omega}=\overbarB{\mathcal{B}}_{\omega} \circ (\pi \times \pi)$. By construction, $\overbarB{\mathcal{B}}_{\omega}$ is non-degenerate. Then, there is a 2-cocycle $\bar{\omega}$ on $M/\Rad(\omega)$ such that $\B_{\bar{\omega}}=\overbarB{\mathcal{B}}_{\omega}.$ The cohomology class of $\omega$ satisfies $[\omega] = [\bar{\omega} \circ (\pi \times \pi)]$. This shows that, up to coboundary, any $2$-cocycle $\omega$ on $M$ is inflated from a unique non-degenerate $2$-cocycle $\bar{\omega}$ on $M/\Rad(\omega)$. \par \vspace{2pt}

On the other hand, by \cite[Proposition 2.1.1, p. 14]{Ka1} any $2$-cocycle on a finite group with values in $\hspace{3pt}\overbarK{K}^{\times}\hspace{0pt}$ is cohomologous to a cocycle with values in a cyclotomic ring of integers. So, if one starts with a $2$-cocycle $\omega$ on $M$ with values in $K^{\times}$, the process of inflating from $M/\Rad(\omega)$ described before can be indeed achieved in a cyclotomic field extension of $K$. \par \vspace{2pt}

Consider now the twisted group algebra $K^{\omega}[M]=\oplus_{m \in M} K\hspace{-1pt} u_m$, where the product is given by $u_m u_{m'}=\omega(m,m')u_{mm'}$. Assume that $K$ is large enough so that $K^{\omega}[M]$ splits as an algebra. The center of $K^{\omega}[M]$ is spanned, as a vector space, by the set $\{u_m: m \in \Rad(\omega)\}$. Write, for short, $\overbarM{M}=M/\Rad(\omega)$. Suppose that $\omega$ is inflated from a non-degenerate $2$-cocycle $\bar{\omega}$ on \hspace{2pt}$\overbarM{M}$ (we extend $K$ to achieve this if necessary). Then, the twisted group algebra $K^{\bar{\omega}}[\hspace{3.5pt}\overbarM{M}\hspace{2pt}]$ is a matrix algebra and there is surjective algebra map $K^{\omega}[M] \twoheadrightarrow K^{\bar{\omega}}[\hspace{3.5pt}\overbarM{M}\hspace{2pt}]$. Let $V$ be the unique (up to isomorphism) irreducible representation of $K^{\bar{\omega}}[\hspace{3.5pt}\overbarM{M}\hspace{2pt}]$. By inflation, $V$ is also an irreducible representation of $K^{\omega}[M].$ \par \vspace{2pt}

The following lemma can be easily proved:

\begin{lemma}\label{irrtwgpal}
Let $\widehat{M}$ denote the group of characters of $M$ with values in $K^{\times}$. For every $\phi \in \widehat{M}$, let $V_{\phi}$ be the representation of $K^{\omega}[M]$ which is equal to $V$ as a\vspace{1pt} vector space and whose action is defined by $u_m \hspace{1.5pt} \scalebox{0.82}{$\odot$}\hspace{1.5pt} v = \phi(m) (u_m \cdot v).$ Then:
\begin{enumerate}
\item[{\it (i)}] $V_{\phi}$ is an irreducible representation of $K^{\omega}[M]$. \vspace{2pt}
\item[{\it (ii)}] Every irreducible representation of $K^{\omega}[M]$ is isomorphic to $V_{\phi}$ for some \linebreak $\phi \in \widehat{M}$. \vspace{2pt}
\item[{\it (iii)}] $V_{\phi} \simeq V_{\psi}$ if and only if $\phi \vert_{\Rad(\omega)}=\psi \vert_{\Rad(\omega)}$. \vspace{2pt}
\item[{\it (iv)}] The irreducible representations of $K^{\omega}[M]$ are in one-to-one correspondence with the irreducible representations of $\Rad(\omega)$.
\item[{\it (v)}] The dimension of every irreducible representation of $K^{\omega}[M]$ equals $\sqrt{\hspace{-0.5pt}\big\vert\hspace{-0.5pt}\frac{M}{\Rad(\omega)}\hspace{-0.5pt}\big\vert\hspace{-0.5pt}}$.
\item[{\it (vi)}] The character $\chi_{\phi}:K^{\omega}[M] \rightarrow K$ afforded by $V_{\phi}$ is given by:
$$\chi_{\phi}(u_m) =
\begin{cases}
\sqrt{\big\vert\frac{M}{\Rad(\omega)}\big\vert}\hspace{2pt} \phi(m) & \text{ if } m \in \Rad(\omega), \\
\hspace{1cm}0 & \text{ otherwise. }
\end{cases}$$
\end{enumerate}
\end{lemma}

\subsection{Irreducible cocharacters of $(K\hspace{-0.9pt}G)_J$} The Hopf algebra $(K\hspace{-0.9pt}G)_J$ is cosemisimple by \cite[Corollary 3.6]{AEGN}. The irreducible corepresentations of $(K\hspace{-0.9pt}G)_J$ were determined by Etingof and Gelaki in \cite[Section 3]{EG1}. The following result is \cite[Proposition 2.1]{CM3}. It reinterprets in our setting and summarizes \cite[Propositions 3.1, 4.1, and 4.2]{EG1}.

\begin{proposition}\label{decomp}
Let $\{\tau_{\ell}\}_{\ell=1}^n$ be a set of representatives of the double cosets of $M$ in $G$. Then:
\begin{enumerate}
\item[{\it (i)}] As a coalgebra, $(K\hspace{-0.9pt}G)_J$ decomposes as the direct sum of subcoalgebras
\begin{equation}\label{decompKGJ}
(K\hspace{-0.9pt}G)_J = \bigoplus_{\ell=1}^n K(M\tau_{\ell} M). \vspace{2pt}
\end{equation}

\item[{\it (ii)}] The subcoalgebra $K(M\tau_{\ell} M)$ has a basis given by $\{e_{\phi}\tau_{\ell}e_{\psi}\}_{(\phi,\psi)\in N_{\tau_{\ell}}}$, where \vspace{1pt}
$$\hspace{1cm} N_{\tau_{\ell}}=\big\{(\phi,\psi) \in \widehat{M}\times\widehat{M} \ : \ \psi(m)=\phi(\tau_{\ell}^{-1}m\tau_{\ell}) \hspace{7pt} \forall m \in M\cap (\tau_{\ell} M \tau_{\ell}^{-1}) \big\}.\vspace{1pt}$$
(Notice that if $M \cap (\tau_{\ell} M \tau_{\ell}^{-1})=\{1\}$, then $N_{\tau_{\ell}}=\widehat{M} \times\widehat{M}$.) \vspace{6pt}

\item[{\it (iii)}] The dual algebra of $K(M\tau_{\ell} M)$ is isomorphic to the twisted group algebra $K^{(\omega,\omega^{-1})\vert _{N_{\tau_{\ell}}}}[N_{\tau_{\ell}}]$.
\end{enumerate}
\end{proposition}

In view of the decomposition \eqref{decompKGJ}, to describe the irreducible cocharacters of $(K\hspace{-0.9pt}G)_J$, it suffices to restrict our attention to those of $K(M\tau_{\ell} M)$. Set, for short, $\tau=\tau_{\ell}$. Denote by $\Rad_{\tau}$ the radical of the restriction of $(\omega,\omega^{-1})$ to $N_{\tau}$.

\begin{proposition}\label{Prop:irr} For every $m,m'\in M,$ consider the element
\begin{equation}\label{irrcoc}
c_{\tau}(m,m'):=\sqrt{\Big\vert\frac{N_{\tau}}{\Rad_{\tau}}\Big\vert}\sum_{(\phi,\psi)\in\Rad_{\tau}}\phi(m)\psi(m') e_{\phi} \tau e_{\psi}.
\end{equation}
Then:
\begin{enumerate}
\item[{\it(i)}]  $c_{\tau}(m,m')$ is an irreducible cocharacter of $(K\hspace{-0.9pt}G)_J$. \vspace{2pt}
\item[{\it (ii)}] $c_{\tau}(m_1,m'_1) = c_{\tau'}(m_2,m'_2)$ if and only if $\tau=\tau'$ and $(m_1,m'_1)(m_2,m'_2)^{-1} \in (\Rad_{\tau})^{\perp}.$
\end{enumerate}
\end{proposition}

\begin{proof}
Let $\{u_{(\phi,\psi)}\}_{(\phi,\psi)\in N_{\tau}}$ be the dual basis of $\{e_{\phi}\tau e_{\psi}\}_{(\phi,\psi)\in N_{\tau}}$. One can check that:
$$u_{(\phi_1,\psi_1)} u_{(\phi_2,\psi_2)}=\omega(\phi_1,\phi_2) \omega^{-1}(\psi_1,\psi_2)u_{(\phi_1\phi_2,\psi_1\psi_2)}.$$
We identify $K(M\tau M)^*$ with the twisted group algebra $K^{(\omega,\omega^{-1})\vert_{N_{\tau}}}[N_{\tau}]$. By duality, the irreducible cocharacters of $K(M\tau M)$ are obtained as the irreducible characters of $K(M\tau M)^*$. \par \smallskip

(i) The elements of $M\times M$ can be viewed\vspace{1pt} as characters on $\widehat{N_{\tau}}$ in the natural way. For every pair $(m,m') \in M \times M$, consider\vspace{-1pt} the irreducible representation $V_{(m,m')}$ of $K^{(\omega,\omega^{-1})\vert_{N_{\tau}}}[N_{\tau}]$ and the corresponding character $\chi_{(m,m')}$ as\vspace{1.5pt} in Lemma \ref{irrtwgpal}. Every irreducible cocharacter of $K(M\tau M)$ will be of this form.\vspace{1.5pt} We describe $\chi_{(m,m')}$ as an element in $K(M\tau M)$ by using Lemma \ref{irrtwgpal}(vi):
$$\begin{array}{rl}
\chi_{(m,m')} & \hspace{-4.5pt}=\hspace{1.5pt} {\displaystyle \sum_{(\phi,\psi)\in N_{\tau}} \chi_{(m,m')}(u_{(\phi,\psi)})e_{\phi} \tau e_{\psi}}  \vspace{5pt} \\
              & \hspace{-4.5pt}=\hspace{1.5pt} {\displaystyle \sqrt{\Big\vert\frac{N_{\tau}}{\Rad_{\tau}}\Big\vert} \sum_{(\phi,\psi)\in \Rad_{\tau}} \phi(m)\psi(m') e_{\phi} \tau e_{\psi}} \vspace{5pt} \\
              & \hspace{-4.5pt}=\hspace{1.5pt} c_{\tau}(m,m').
\end{array}$$
\smallskip

(ii) According to Lemma \ref{irrtwgpal}(iii), $V_{(m_1,m'_1)} \simeq V_{(m_2,m'_2)}$ if and only if $(m_1,m'_1) \vert_{\Rad_{\tau}}=(m_2,m'_2) \vert_{\Rad_{\tau}}$. This condition\vspace{0.5pt} is equivalent to $(m_1,m'_1)(m_2,m'_2)^{-1} \in (\Rad_{\tau})^{\perp}.$ Finally, notice\vspace{0.66pt} that $c_{\tau}(m_1,m'_1) \neq c_{\tau'}(m_2,m'_2)$ if $\tau \neq \tau'$ because\vspace{0.5pt} the intersection of $K(M\tau M)$ and $K(M\tau' M)$ is trivial in view of the decomposition \eqref{decompKGJ} of $(K\hspace{-0.9pt}G)_J$.
\end{proof}

Let $\{(m_i,m'_i)\}_{i=1}^r$ be coset representatives of $(\Rad_{\tau})^{\perp}$ in $M \times M$. By using that $\phi(e_{\varepsilon})=\delta_{\phi,\varepsilon}$ and $(\varepsilon,\varepsilon)\in \Rad_{\tau}$, we have the following chain of equalities:
$$\begin{array}{rl}
e_{\varepsilon} \tau e_{\varepsilon} & \hspace{-4.5pt}=\hspace{1.5pt} {\displaystyle \sum_{(\phi,\psi)\in \Rad_{\tau}}  \phi(e_{\varepsilon}) \psi(e_{\varepsilon}) e_{\phi} \tau e_{\psi}} \vspace{5pt} \\
& \hspace{-4.5pt}=\hspace{1.5pt} {\displaystyle \frac{1}{\vert M \vert^2} \sum_{(\phi,\psi)\in \Rad_{\tau}} \hspace{2pt} \sum_{m,m' \in M} \phi(m)\psi(m') e_{\phi} \tau e_{\psi}} \vspace{5pt} \\
& \hspace{-4.5pt}=\hspace{1.5pt} {\displaystyle  \frac{1}{\vert M \vert^2} \sum_{(\phi,\psi)\in \Rad_{\tau}} \big\vert (\Rad_{\tau})^{\perp} \big\vert  \sum_{i=1}^r \phi(m_i)\psi(m'_i) e_{\phi} \tau e_{\psi}} \vspace{5pt} \\
& \hspace{-4.5pt}=\hspace{1.5pt} {\displaystyle \frac{1}{\vert \Rad_{\tau} \vert} \sum_{i=1}^r \sum_{(\phi,\psi)\in \Rad_{\tau}} \phi(m_i)\psi(m'_i) e_{\phi} \tau e_{\psi}}.
\end{array}$$
Hence:
\begin{equation}\label{S2Eq1}
\sum_{i=1}^r c_{\tau}(m_i,m'_i) = \sqrt{\Big\vert\frac{N_{\tau}}{\Rad_{\tau}}\Big\vert} |\Rad_{\tau}| e_{\varepsilon} \tau e_{\varepsilon} = \sqrt{|N_{\tau}||\Rad_{\tau}|} e_{\varepsilon} \tau e_{\varepsilon}.
\end{equation}

\begin{proposition}
If $\omega$ is non-degenerate, then $\sqrt{|N_{\tau}||\Rad_{\tau}|}$ is a natural number that divides $|M|$.
\end{proposition}

\begin{proof}
Since $\omega$ is non-degenerate, the algebra $K^{(\omega,\omega^{-1})}[\hspace{1pt}\widehat{M} \times \widehat{M}\hspace{1pt}]$ has\vspace{2pt} a unique irreducible representation, say $W$, of\vspace{1pt} dimension $|M|$. Consider $W$ as a representation of $K^{(\omega,\omega^{-1})}[N_{\tau}]$ by restriction of scalars. It decomposes\vspace{1pt} as the direct sum $W \simeq \oplus_{\phi \in \widehat{\Rad_{\tau}}} V_{\phi}^{(s_{\phi})}$. On the other hand, $K^{(\omega,\omega^{-1})}[\hspace{1pt}\widehat{M} \times \widehat{M}\hspace{1pt}]$ is free\vspace{1pt} as a module over $K^{(\omega,\omega^{-1})}[N_{\tau}]$. Bearing in mind\vspace{2pt} that all the $V_{\phi}$'s have the same dimension, the above implies that all the numbers $s_{\phi}$'s are equal. Write\vspace{1pt} simply $s$ for them. Counting dimensions, we obtain:
\begin{equation}\label{S2Eq2}
|M| = s |\widehat{\Rad_{\tau}}| \sqrt{\Big\vert\frac{N_{\tau}}{\Rad_{\tau}}\Big\vert}= s\sqrt{|N_{\tau}||\Rad_{\tau}|}.
\end{equation}
This establishes the claim.
\end{proof}

The following result refines \cite[Proposition 2.2]{CM3}, which required the hypothesis $M \cap \big(\tau M \tau^{-1}\big) =\{1\}$:

\begin{proposition}\label{vipcharacter}
If $\omega$ is non-degenerate, then the cocharacter $|M|e_{\varepsilon} \tau e_{\varepsilon}$ is contained in every Hopf order of $(K\hspace{-0.9pt}G)_J$.
\end{proposition}

\begin{proof}
From \eqref{S2Eq1} and \eqref{S2Eq2} we have:
$$|M|e_{\varepsilon} \tau e_{\varepsilon} = s\sqrt{|N_{\tau}||\Rad_{\tau}|} e_{\varepsilon} \tau e_{\varepsilon} = s \sum_{i=1}^r c_{\tau}(m_i,m'_i).$$
Now, use the fact that the right-hand side term is a cocharacter of $(K\hspace{-0.9pt}G)_J$ and that a cocharacter is contained in any Hopf order by Proposition \ref{character}(ii).
\end{proof}

\section{Lagrangian decomposition and a sufficient condition for the existence of Hopf orders}

In \cite[Proposition 4.1]{CM3} we constructed an example of integral Hopf order for the twisted group algebra $(K\hspace{-1pt}S_4)_J$, where $J$ was a twist arising from a subgroup of $S_4$ isomorphic to the Klein four-group. The key fact in this construction was the existence of a Hopf order $X$ of $K\hspace{-1pt}S_4$ such that $J^{\pm 1} \in X \otimes_R X$. Then, the twisted Hopf $R$-algebra $X_J$ of $X$ provides a Hopf order of $(K\hspace{-1pt}S_4)_J$. The goal of this section is to fit this example in a general group-theoretical framework. \par \vspace{2pt}

We start with a brief discussion on Lagrangian decompositions of an abelian group of central type. We refer the reader to \cite[Section 4]{D} and \cite[Section 1]{BGM} for further details. \par \vspace{2pt}

Let $M$ be a finite abelian group of central type and $\beta:M \times M \rightarrow K^{\times}$ a non-degenerate skew-symmetric pairing. For a subgroup $L$ of $M$, we consider the orthogonal complement
$$L^{\perp} = \{m \in M : \beta(m,l)=1 \hspace{5pt} \forall l \in L\}.$$
The subgroup $L$ is said to be a \emph{Lagrangian} of $M$ if $L=L^{\perp}$. A Lagrangian subgroup gives rise to a short exact sequence
$$1 \rightarrow L \rightarrow M \stackrel{\hspace{-2pt}\pi}{\rightarrow} \widehat{L} \rightarrow 1.$$
Here, $\pi$ is defined\vspace{-1pt} by $\pi(m)(l)=\beta(m,l)$ for all $m \in M, l \in L$. Suppose that $\pi$ splits. Then, $M\simeq L\times \widehat{L}$ and such a decomposition is called\vspace{0.5pt} \emph{Lagrangian decomposition} of $M$. It is proved in \cite[Lemma 4.2]{D} and \cite[Proposition 1.7]{BGM} that:
\begin{enumerate}
\item A Lagrangian decomposition of $M$ always exists. \vspace{2pt}
\item Writing every element of $M$ as a pair $(l,\lambda)$, with $l \in L,\lambda \in \widehat{L}$, the pairing $\beta$ takes the form:
$$\beta\big((l,\lambda),(l',\lambda')\big) = \lambda(l')\lambda'(l)^{-1}.$$
\end{enumerate}

Let $\alpha: M \times M \rightarrow K^{\times}$ be now a non-degenerate $2$-cocycle. Applying the preceding discussion to the pairing ${\mathcal B}_{\alpha}$, we obtain that $\alpha$ is (up to coboundary) given by:
\begin{equation}\label{eval}
\alpha\big((l,\lambda),(l',\lambda')\big) = \lambda(l').
\end{equation}

We will next see that this formula, when applied to the twisting procedure, allows to express the twist in an illuminating form. \par \vspace{2pt}

Suppose that $\omega: \widehat{M} \times \widehat{M} \rightarrow K^{\times}$ is a non-degenerate $2$-cocycle. Let $\widehat{M} \simeq L\times \widehat{L}$ be a Lagrangian decomposition of $\widehat{M}$ such that $\omega$ is given as in \eqref{eval}. Call $f:L\times \widehat{L} \rightarrow \widehat{M}$ the isomorphism giving the previous decomposition. Identify the character group of $\widehat{M}$ with $M$ in the natural way, and similarly for $\widehat{L}$. Thus, we have an isomorphism $\widehat{f}:M \rightarrow \widehat{L} \times L$. Under these isomorphisms, we can see the elements of $\widehat{M}$ as pairs $(l,\lambda)$, with $l \in L,\lambda \in \widehat{L}$, and the elements of $M$ as pairs $(\lambda',l')$, with $\lambda' \in \widehat{L},l' \in L$. The evaluation\vspace{1pt} of $\widehat{M}$ at $M$ is then given by $(l,\lambda)(\lambda',l')=\lambda'(l)\lambda(l')$. \par \vspace{2pt}

Under these identifications, we also have that:
\begin{enumerate}
\item[(1)] The primitive idempotents in $K\hspace{-1pt}M$ are of the form $e_{(l,\lambda)}$. \vspace{2pt}
\item[(2)] Each $e_{(l,\lambda)}$ can be rewritten as the product $e_{\lambda} e_l$, where now $e_{\lambda}\in K\hspace{-1pt}L$ and $e_l \in K\hspace{-0.5pt}\widehat{L}$. (We again view $l \in L$ as a character of $\widehat{L}$.)
\end{enumerate}

The next result ties the twist afforded by $\omega$ with the dual bases $\{(\lambda,e_{\lambda})\}_{\lambda \in \widehat{L}}$ and $\{(e_l,l)\}_{l \in L}$ of $K\hspace{-1pt}L$:

\begin{lemma}\label{lemJ}
The twist $J$ in $K\hspace{-1pt}M\otimes K\hspace{-1pt}M$ afforded by $\omega$ can be expressed as:
\begin{equation}\label{Jrwt}
J \hspace{1pt}=\hspace{1pt} \sum_{\lambda \in \widehat{L}} e_{\lambda} \otimes \lambda \hspace{1pt}=\hspace{1pt} \sum_{l \in L} l \otimes e_l.\vspace{-2pt}
\end{equation}
Hence, $J$ lies in $K\hspace{-1pt}L\otimes K\hspace{-0.5pt}\widehat{L}.$
\end{lemma}

\begin{proof}
We compute:
$$\begin{array}{rl}
J & \hspace{-2pt}=\hspace{3pt} {\displaystyle \sum_{l,l' \in L} \hspace{2pt} \sum_{\lambda,\lambda' \in \widehat{L}} \omega\big((l,\lambda),(l',\lambda')\big) e_{(l,\lambda)} \otimes e_{(l',\lambda')}} \vspace{4pt} \\
  & \hspace{-6pt}\stackrel{\eqref{eval}}{=}\hspace{-1pt} {\displaystyle \sum_{l,l' \in L} \hspace{2pt} \sum_{\lambda,\lambda' \in \widehat{L}} \lambda(l') e_{\lambda} e_l \otimes e_{\lambda'} e_{l'}}  \vspace{5pt} \\
  & \hspace{-2pt}=\hspace{3pt} {\displaystyle \sum_{\lambda \in \widehat{L}} e_{\lambda} \bigg(\sum_{l \in L} e_l \bigg) \otimes \bigg(\sum_{\lambda' \in \widehat{L}} e_{\lambda'} \bigg) \bigg(\sum_{l' \in L} \lambda(l') e_{l'} \bigg)}  \vspace{5pt} \\
  & \hspace{-2pt}=\hspace{3pt} {\displaystyle \sum_{\lambda \in \widehat{L}} e_{\lambda} \otimes \lambda.}
\end{array}$$

The expression in the right-hand side of \eqref{Jrwt} is obtained in a similar way by using that $\sum_{\lambda \in \widehat{L}} \lambda(l')e_{\lambda}=l'$.
\end{proof}

Observe that
\begin{equation}\label{J-1rwt}
J^{-1} \hspace{1pt}=\hspace{1pt} \sum_{\lambda \in \widehat{L}} e_{\lambda} \otimes \lambda^{-1} \hspace{1pt}=\hspace{1pt} \sum_{l \in L} l^{-1} \otimes e_l.
\end{equation}

We are now in a position to state the main result of this section:

\begin{theorem}\label{thm:main1}
Let $K$ be a (large enough) number field with ring\vspace{1pt} of integers $R$. Let $G$ be a finite group and $M$ an abelian subgroup of $G$ of central type. Consider the twist \linebreak $J$ in $K\hspace{-1pt}M \otimes K\hspace{-1pt}M$ afforded by a non-degenerate $2$-cocycle $\omega: \widehat{M} \times \widehat{M} \rightarrow K^{\times}$. \par \vspace{2pt}

Fix a Lagrangian decomposition $\widehat{M} \simeq L \times \widehat{L}$. Suppose that $L$ (viewed as inside of $M$) is contained in a normal abelian subgroup $N$ of $G$. Then, $(K\hspace{-0.8pt}G)_J$ admits a Hopf order over $R$.
\end{theorem}

\begin{proof}
We will construct a Hopf order $X$ of $K\hspace{-0.8pt} G$ such that $J^{\pm 1} \in X \otimes_R X$. Then, $X$ will be closed under the coproduct and the antipode of $(K\hspace{-0.9pt} G)_J$ by Proposition \ref{twistorder}. \par \vspace{2pt}

Let $X$ be\vspace{-1pt} the $R$-subalgebra of $K\hspace{-0.8pt} G$ generated by the set $\{e_{\nu}g: \nu \in \widehat{N}, g\in G\}$. Since $N$ is normal, $G$ acts on $\widehat{N}$ by $(g \triangleright \nu)(n)=\nu(g^{-1}ng)$ for all $n \in N$. Thus, $ge_{\nu}g^{-1}=e_{g \hspace{0.65pt}\triangleright\hspace{0.65pt} \nu}$, and we get the rule: $(e_{\nu}g)(e_{\nu'}g')=e_{\nu}e_{g \hspace{0.65pt}\triangleright\hspace{0.65pt} \nu'}gg'$. Choose\vspace{1.5pt} a set $Q$ of coset representatives of $N$ in $G$ with $1$ as the representative of $N$. Then, $X$ is a free $R$-module with basis $\{e_{\nu}q: \nu \in \widehat{N}, q \in Q\}$. One can see\vspace{1pt} that $X$ is a Hopf order of $K\hspace{-0.8pt} G$ by using the formulas:  $\Delta(e_{\nu})=\sum_{\eta \in \widehat{N}} e_{\eta} \otimes e_{\eta^{-1}\nu}, \hspace{3pt} \varepsilon(e_{\nu})=\delta_{\nu,1},$ and $S(e_{\nu})=e_{\nu^{-1}}$. \par \vspace{2pt}

An idempotent of $K\hspace{-1pt}L$ is a sum of primitive\vspace{0.5pt} idempotents of $K\hspace{-1pt}N$ because $K\hspace{-1pt}L$ is a subalgebra of $K\hspace{-1.2pt}N$ and $K\hspace{-1pt}N=\oplus_{\nu \in \widehat{N}} \hspace{1pt} K\hspace{-1pt} e_{\nu}$. Hence, $X$ contains all the primitive idempotents of $K\hspace{-1pt}L$. The\vspace{1.5pt} expressions \eqref{Jrwt} and \eqref{J-1rwt} yield that $J^{\pm 1} \in X \otimes_R X$.
\end{proof}

A particular case in which the hypothesis of Theorem \ref{thm:main1} is satisfied is when $M$ itself is contained in a normal abelian subgroup of $G$. Although this does not always happen; see Example \ref{example1}. \par \vspace{2pt}

The next characterization will be useful to prove that, in some situations, $X$ is the unique Hopf order. A warning on notation is first necessary. Since we will simultaneously work with the primitive idempotents of different group algebras, we will specify the group as a superscript to distinguish them.

\begin{proposition}\label{cond:unique}
Retain the hypotheses of Theorem \ref{thm:main1}. Assume furthermore that the natural action of $G/N$ on $N$ induced by conjugation is faithful (equivalently, that $C_G(N)=N$). Consider the Hopf order $X$ of $(K\hspace{-0.69pt}G)_J$ over $R$ constructed above. Let $Y$ be a Hopf order of $(K\hspace{-0.69pt}G)_J$ over $R$. The following assertions are equivalent: \vspace{2pt}
\begin{enumerate}
\item[{\it (i)}] $X=Y$. \vspace{2pt}
\item[{\it (ii)}] $e_{\nu}^N \in Y$ for all $\nu \in \widehat{N}$. \vspace{4pt}
\item[{\it (iii)}] $e_{\varepsilon}^L, e_{\varepsilon}^N \in Y$.
\end{enumerate}
\end{proposition}

\begin{proof}
(i) $\Rightarrow$ (ii) This is clear by the construction of $X$. \par \vspace{2pt}

(ii) $\Rightarrow$ (iii) The\vspace{0.5pt} map from $\widehat{N}$ to $\widehat{L}$ given by restriction is surjective, and for every $\lambda \in \widehat{L}$, we have:
$$e^L_{\lambda} \hspace{2pt}=\hspace{1pt} \sum_{\begin{subarray}{c} \nu \in \widehat{N} \vspace{1pt} \\ \nu \vert_L = \lambda \end{subarray}} e^N_{\nu}.$$
In particular, $e_{\varepsilon}^L \in Y$. \par \vspace{2pt}

(iii) $\Rightarrow$ (i) We first\vspace{0.5pt} show that $X \subseteq Y$. By Proposition \ref{sub}, $Y \cap (K\hspace{-1pt}M)$ is a Hopf order of $K\hspace{-1pt}M$ over $R$. Proposition \ref{character} yields that $M$ is contained in $Y$. Similarly, $Y \cap (K\hspace{-1pt}L)$ is a Hopf order of $K\hspace{-1pt}L$ over $R$. Notice\vspace{-1pt} that $e_{\lambda}^L=(\lambda^{-1} \otimes id)\Delta(e_{\varepsilon}^L)$ for every $\lambda \in \widehat{L}$. Using Proposition \ref{character} again, we get that $e_{\lambda}^L \in Y$. In\vspace{1pt} view\vspace{-0.5pt} of Equations \ref{Jrwt} and \ref{J-1rwt}, $J$ and $J^{-1}$ belong to $Y \otimes_R Y$. So, $Y$ is\vspace{0.5pt} also a Hopf order of $K\hspace{-1pt}G$ over $R$. By applying once more Proposition \ref{character}, we obtain\vspace{1pt} that $G$ is contained in $Y$. Now, $Y \cap (K\hspace{-1pt}N)$ is a Hopf order of $K\hspace{-1pt}N$ over $R$. Arguing\vspace{-0.5pt} as before, $e_{\nu}^N=(\nu^{-1} \otimes id)\Delta(e_{\varepsilon}^N)$ belongs to $Y$ for every $\nu \in \widehat{N}$. Hence, $X \subseteq Y$. \par \vspace{2pt}

We next show that $Y \subseteq X$. Pick an arbitrary element $y \in Y$ and express it in the following form:
$$y \hspace{3pt}=\hspace{2pt}\sum_{\nu\in\widehat{N}\hspace{-1pt},\hspace{0.5pt}q \in Q} k_{\nu,q} e_{\nu}q, \hspace{2pt}\textrm{ with } k_{\nu,q}\in K.$$
We will prove that $k_{\nu,q} \in R$ for all $\nu \in \widehat{N},q \in Q$. This task admits two reductions. First, by\vspace{1pt} multiplying $y$ with each  primitive idempotent in $K\hspace{-1pt}N$ (which belong to $Y$), it suffices to show that if $\sum_{q \in Q} k_q e _{\nu} q \in Y$, with $k_q \in K$, then $k_q \in R$. Second, by multiplying this new element with the elements in $Q$ (which also belong to $Y$), it suffices to show that $k_1 \in R$. \par \vspace{2pt}

Let $\{\mu_i\}_{i=1}^t$ be\vspace{0.75pt} a set of generators of $\widehat{N}$ as an abelian group. Since $Y$ is a Hopf order of $K\hspace{-1pt}G$ that contains\vspace{1pt} $\{e_{\nu}^N\}_{\nu \in \widehat{N}}$, the element resulting from the following computation will be in $Y \otimes_R \ldots \otimes_R Y$ ($t+1$ times):
$$\begin{array}{l}
{\displaystyle \big(e^N_{\mu_1}\otimes \ldots \otimes e^N_{\mu_t} \otimes 1\big)\Delta^{(t)}\Big(\sum_{q\in Q} k_q e^N_{\nu}q\Big)\big(e^N_{\mu_1}\otimes \ldots \otimes e^N_{\mu_t}\otimes 1\big)} \vspace{6pt} \\
 \hspace{0.8cm}= \hspace{0.1cm} {\displaystyle \sum_{q\in Q} \big(e^N_{\mu_1}\otimes \ldots \otimes e^N_{\mu_t} \otimes 1\big) \bigg(\sum_{\eta_1,\ldots,\eta_t \in \widehat{N}} k_q e^N_{\eta_1}q \otimes \ldots \otimes e^N_{\eta_t}q \otimes e^N_{\nu-\eta_1-\ldots - \eta_t}q\bigg)} \vspace{4pt} \\
  {\displaystyle  \hspace{10.1cm} \big(e^N_{\mu_1}\otimes \ldots \otimes e^N_{\mu_t}\otimes 1\big)} \vspace{8pt} \\
 \hspace{0.8cm}= \hspace{0.1cm} {\displaystyle \sum_{q\in Q} \hspace{0.6mm} \sum_{\eta_1,\ldots, \eta_t \in \widehat{N}\hspace{-1pt}} k_q e^N_{\mu_1} e^N_{\eta_1} e^N_{q \hspace{1pt}\triangleright \mu_1} q \otimes \ldots \otimes e^N_{\mu_t}e^N_{\eta_t} e^N_{q \hspace{1pt}\triangleright \mu_t} q \otimes e^N_{\nu-\eta_1-\ldots -\eta_t}q.}
\end{array}$$

The idempotents $e^N_{\mu_{i}}$'s are orthogonal.\vspace{0.5pt} The only non-zero summands in the above sum will be those in which
$\mu_i=\eta_i= q \triangleright \mu_i$ for every $i=1,\ldots, t$. By hypothesis, the action of $G/N$ on $N$ is faithful. So, the action of $G/N$ on $\widehat{N}$ is faithful\vspace{-0.5pt} as well. Then, $q \triangleright \mu_i=\mu_i$ for all $i$ if and only if $q=1$. The above sum reduces to:
$$k_1 e_{\mu_1}^N \otimes \ldots \otimes e_{\mu_t}^N \otimes e_{\nu-\mu_1-\ldots-\mu_t}^N.$$
This element belongs to $Y \otimes_R \ldots \otimes_R Y$ ($t+1$ times). Hence, it satisfies a monic polynomial with coefficients in $R$. Since the tensor product of idempotents is an idempotent, $k_1$ satisfies a monic polynomial with coefficients in $R$. Therefore, $k_1 \in R$ and we are done.
\end{proof}

We can squeeze a bit more condition (iii) in Proposition \ref{cond:unique}:

\begin{proposition}\label{squeeze}
Retain the hypotheses of Theorem \ref{thm:main1}. In addition, suppose that $L$ generates $N$ as a $G/N$-module. Let $Y$ be a Hopf order of $(K\hspace{-0.69pt}G)_J$ over $R$. Then, $e_{\varepsilon}^L \in Y$ if and only if $e_{\varepsilon}^N \in Y$.
\end{proposition}

\begin{proof}
Keep\vspace{2pt} the notation of the proof of Theorem \ref{thm:main1}. Bear in mind that $e_{\varepsilon}^N$ is an integral in $K\hspace{-1pt}N$ such that $\varepsilon(e_{\varepsilon}^N)=1$. For $g \in G$, notice that $g \triangleright e_{\varepsilon}^L = ge_{\varepsilon}^Lg^{-1}=e_{\varepsilon}^{gLg^{-1}}$ and $\varepsilon\big(e_{\varepsilon}^{gLg^{-1}}\big)=1$. Consider\vspace{0.5pt} the element $\Lambda:= \prod_{g \in G} e_{\varepsilon}^{gLg^{-1}}$ in $K\hspace{-1pt}N$. That $L$ generates $N$ as a $G/N$-module means\vspace{2pt} that there is a subset $F$ of $Q$ such that $N=\prod_{g \in F}\hspace{1pt} gLg^{-1}$. This implies that $\Lambda$ is an integral in $K\hspace{-1pt}N$. Moreover, $\varepsilon(\Lambda)=1$. By the uniqueness of integrals, it must be $\Lambda=e_{\varepsilon}^N$. \par \vspace{2pt}

Suppose now that $e_{\varepsilon}^L \in Y$. We have\vspace{1pt} seen in the proof of (iii) $\Rightarrow$ (i) in Proposition \ref{cond:unique} that if $e_{\varepsilon}^L \in Y$, then $Y$ is a Hopf order of $K\hspace{-0.69pt}G$\vspace{1pt} and $G \subset Y$. Therefore, $ge_{\varepsilon}^Lg^{-1} \in Y$ and $e_{\varepsilon}^N=\prod_{g \in G} \hspace{0.5pt} ge_{\varepsilon}^Lg^{-1} \in Y$. \par \vspace{2pt}

Conversely, suppose that $e_{\varepsilon}^N \in Y$. Arguing as in the proof of (iii) $\Rightarrow$ (i) in Proposition \ref{cond:unique} with $Y \cap (K\hspace{-1pt}N)$, we get that $e_{\nu}^N=(\nu^{-1} \otimes id)\Delta(e_{\varepsilon}^N)$ belongs to $Y$ for every $\nu \in \widehat{N}$. The same argument of the proof of (ii) $\Rightarrow$ (iii) gives that $e_{\varepsilon}^L \in Y$. \par \vspace{2pt}

Notice that the parts of the proof of Proposition \ref{cond:unique} we just invoked do not require that the action of $G/N$ on $N$ is faithful.
\end{proof}

The following example shows a finite non-abelian group $G$ containing an abelian subgroup $M$ of central type such that:
\begin{enumerate}
\item[(1)] $M$ is not included in a normal abelian subgroup of $G$. \vspace{2pt}
\item[(2)] There is a Lagrangian decomposition $M \simeq L \times \widehat{L}$ such that $L$ is however included in a normal abelian subgroup of $G$.
\end{enumerate}

\begin{example}\label{example1}
Let $p$ be a prime number and $\F_p$ the field with $p$ elements. Consider the subgroup $G$ of $\GL_{2n+2}(\F_p)$ consisting of matrices $(a_{ij})$ defined by the following conditions:
$$a_{11}=a_{2n+2\hspace{1.5pt} 2n+2}=1 \text{ and } a_{i\hspace{0.25pt} 1} = a_{2n+2\hspace{1pt} j}=0 \text{ for } i= 2,\ldots, 2n+2; j=1,\ldots ,2n+1.$$

By forgetting the first and last rows and columns of every matrix, we make $G$ to fit into the following short exact sequence:
$$1\to \Gamma \to G\to \GL_{2n}(\F_p)\to 1.$$
The subgroup $\Gamma$ consists of matrices of the following form: \vspace{2mm}
$$\hspace{1.5cm}\left(\begin{array}{c|ccc|c}
1       & a_{12} & \ldots & a_{1\hspace{0.75pt} 2n+1}            & a_{1\hspace{0.75pt} 2n+2} \\ \hline
0       &         &        &                       & a_{2\hspace{0.75pt} 2n+2} \vspace{-2pt} \\
\vdots  &         &  \hspace{0.3cm}\textrm{{\large Id}}    &        & \vdots \\
0       &         &        &                       & a_{2n+1\hspace{1.5pt} 2n+2} \\ \hline
0       &  0      & \ldots &   0                   & 1
\end{array}\right).\vspace{2mm}$$
Write $V$ for the abelian group $(\F_p^{2n},+)$ and set $V^*= \Hom_{\F_p}(V,\F_p)$. Take a dual basis $\{(x_i,y_i)\}_{i=1}^{2n} \subset  V^* \times V$. We assign to every matrix $(a_{ij})$ in $\Gamma$ the following pair in $V^* \times V:$
$$\big(a_{2\hspace{1pt} 2n+2}\hspace{1pt} x_1+\ldots+a_{2n+1\hspace{1.5pt} 2n+2}\hspace{1pt} x_{2n}\hspace{1pt},\hspace{1pt} a_{12}\hspace{0.5pt}y_1+\ldots +a_{1\hspace{0.75pt} 2n+1}\hspace{0.5pt} y_{2n}\big).$$
This gives us another exact sequence:
$$0\to \F_p \to \Gamma \to V^* \times V\to 0.$$

Write $z$ for a generator of $(\F_p,+)$. Then, $\Gamma$ has the following presentation:
$$\Gamma=\langle x_i,y_i,z : x_i^p=y_i^p=z^p=[x_i,z]=[y_i,z]=[x_i,x_j]=[y_i,y_j]=1, [y_i,x_j]=z^{\delta_{i,j}}\rangle.$$

Here, we view the generators as inside $\Gamma$ as follows: $x_i$ is the elementary matrix with $1$ in the $(i,2n+2)$-entry, $y_i$ is the elementary matrix with $1$ in the $(1,i)$-entry, and $z$ is the elementary matrix with $1$ in the $(1,2n+2)$-entry. \par \vspace{2pt}

Let $M$ be now the subgroup of $G$ generated by $x_{n+1},\ldots, x_{2n},y_1,\ldots,y_n$, which is clearly of central type. Consider a non-degenerate cocycle $\omega: \widehat{M} \times \widehat{M} \rightarrow K^{\times}$ with Lagrangian decomposition defined by
$M=L \times \widehat{L} = \langle x_{n+1},\ldots, x_{2n}\rangle \times \langle y_1,\ldots, y_n\rangle$. \par \vspace{2pt}

We have that $M$ is abelian and it is not contained in a normal abelian subgroup of $G$. (By multiplying with appropriate matrices, one can see that if $M$ were contained in such a group, then $\Gamma$ would be abelian as well, reaching a contradiction.) However, $L$ is contained in the normal abelian subgroup $N:= \langle x_1,\ldots,x_{2n} \rangle \langle z \rangle$. \par \vspace{2pt}

Here, the roles of $L$ and $\widehat{L}$ can be interchanged. We also have that $\widehat{L}$ is contained in the normal abelian subgroup $\langle y_1,\ldots,y_{2n} \rangle \langle z \rangle$. Notice that the two Hopf orders constructed from each one of these normal subgroups are different. This can be seen when trying to express a primitive idempotent of $K\hspace{-1pt}\widehat{L}$ as an $R$-linear combination of the basis $\{e_{\nu}q: \nu \in \widehat{N}, q \in Q\}$.
\end{example}

All examples of integral Hopf orders in twisted group algebras that we know so far are constructed as in the proof of Theorem \ref{thm:main1}. This suggests the following question:

\begin{question}
Let $K$ be a number\vspace{1pt} field with ring of integers $R$. Let $G$ be a finite group and $J$ a twist for $K\hspace{-0.69pt} G$ arising from an abelian subgroup $M$ and a non-degenerate $2$-cocycle on $\widehat{M}$ with values in $K^{\times}$. Suppose that $(K\hspace{-0.69pt} G)_J$ admits a Hopf order $X$ over $R$. Is there a cohomologous twist $\tilde{J}$ to $J$ such that $\tilde{J}^{\pm 1} \in X \otimes_R X$?
\end{question}

The example of Hopf order $X$ of $K\hspace{-1pt} S_4$ in \cite[Proposition 4.1]{CM3} shows that the initial twist $J$ used there does not satisfy $J^{\pm 1} \in X \otimes_R X$, see \cite[Remark 4.2]{CM3}. However, $J$ could be replaced by a cohomologous twist to achieve that condition.

\section{Existence and uniqueness of Hopf orders in twists of certain semidirect products}\label{ex-uniq}

This section also grows out of the above-mentioned example of integral Hopf order for $(K\hspace{-1pt}S_4)_J$. Such a Hopf order had the additional property of being unique. In this section, we examine this property in the framework defined by the hypotheses of Theorem \ref{thm:main1}. For semidirect products of groups, we provide several conditions that guarantee the uniqueness of the Hopf order constructed there. \par \vspace{2pt}

Let $M$ be a finite abelian group. Suppose that $M=LP$, where $L$ and $P$ are subgroups of $M$ such that $L\cap P=\{1\}$ and $L \simeq P$. Fix an isomorphism $f:P \rightarrow \widehat{L}$. It induces a non-degenerate\vspace{-0.75pt} skew-symmetric pairing $\beta_f:\widehat{L} \times \widehat{P} \to K^{\times}$ given by $\beta_f(\lambda,\rho)= \rho(f^{-1}(\lambda))$. Identifying $\widehat{M}$ with $\widehat{P}\widehat{L}$, we get\vspace{1.5pt} the following $2$-cocycle on $\widehat{M}$:
$$\omega(\rho_1\lambda_1,\rho_2\lambda_2)=\beta_f(\lambda_1,\rho_2), \hspace{10pt} \lambda_i \in \widehat{L}, \rho_i \in \widehat{P} \hspace{1pt}\textrm{ for }\hspace{1pt} i=1,2.\vspace{4pt}$$

The twist $J$ in $K\hspace{-1pt} M \otimes K\hspace{-1pt} M$ afforded by $\omega$ takes the following form:
$$J=\sum_{\lambda \in\widehat{L}} \hspace{1pt}\sum_{\rho \in \widehat{P}}\hspace{1pt} \omega(\lambda,\rho)e_{\lambda}\otimes e_{\rho}.$$
We call $J$ the \emph{twist arising from} $f$. The isomorphism $\widehat{f}:L \rightarrow \widehat{P}$ yields a Lagrangian decomposition $\widehat{M} \simeq L \times \widehat{L}$ such that $J$ can be expressed as in \eqref{Jrwt}. \par \vspace{2pt}

The following result, encompassed in Theorem \ref{thm:main1}, supplies more examples of integral Hopf orders in twisted group algebras:

\begin{theorem}\label{thm:main2}
Let $K$ be a (large enough) number field with ring of integers $R$. Consider the semidirect product $G := N \rtimes Q$ of two finite groups $N$ and $Q$, with $N$ abelian. Let $L$ and $P$ be abelian subgroups of $N$ and $Q$, respectively. Set $M=LP$. Let $\tau \in Q$. Suppose that $N,Q,L,P,$ and $\tau$ satisfy the following conditions:
\begin{enumerate}
\item[{\it (i)}] $L$ and $P$ are isomorphic and commute with one another. \vspace{2pt}
\item[{\it (ii)}] $Q$ acts on $N$ faithfully. \vspace{2pt}
\item[{\it (iii)}] $N=L \oplus (\tau \cdot L)$, where $N$ is written additively. \vspace{2pt}
\item[{\it (iv)}] $N^{\tau} \neq \{1\}$. \vspace{2pt}
\item[{\it (v)}] $\widehat{N}^{\sigma \tau}=\big(\widehat{N}^{\tau}\big) \cap \big(\widehat{N}^{\sigma \tau \sigma^{-1}}\big)=\{\varepsilon\}$ for every $\sigma \in P$ with $\sigma \neq 1$.
\end{enumerate}
Let $J$ be the twist in $K\hspace{-1pt}M \otimes K\hspace{-1pt}M$ arising from an isomorphism $f:P \rightarrow \widehat{L}$. Then, $(K\hspace{-0.9pt}G)_J$ admits a unique Hopf order over $R$. This Hopf order is generated, as an $R$-subalgebra, by the primitive idempotents of $K\hspace{-1.1pt}N$ and the elements of $Q$.
\end{theorem}

\begin{proof}
We argued above that $L$ and $f$ provide a Lagrangian decomposition\linebreak $\widehat{M} \simeq L \times \widehat{L}$ such that $J$ can be expressed as in \eqref{Jrwt}. By hypothesis, $L$ is contained in $N$, and $N$ is a normal abelian subgroup of $G$. Theorem \ref{thm:main1} gives that $(K\hspace{-0.9pt} G)_J$ admits a Hopf order $X$ over $R$. The existence is so ensured. Notice that, in this setting, the Hopf order $X$ constructed in the proof of Theorem \ref{thm:main1} is the $R$-submodule generated by the set $\{e_{\nu}^N q: \nu \in \widehat{N},q \in Q\}$. This set is a basis of $K\hspace{-0.9pt} G$ as a $K$-vector space. \par \vspace{2pt}

For the uniqueness, let $Y$ be a Hopf order of $(K\hspace{-0.9pt} G)_J$ over $R$. The idea of the proof is to establish that $e_{\varepsilon}^L$ belongs to $Y$ and then get that $X=Y$ by applying Propositions \ref{cond:unique} and \ref{squeeze}. (Hypotheses (ii) and (iii) are needed to apply such propositions.) Showing that $e_{\varepsilon}^L \in Y$ will require much technical work. \par \vspace{2pt}

Denote\vspace{-1pt} by $V$ the representation $\text{Ind}_Q^G(K)$ modulo the trivial representation $K$. We identify $\text{Ind}_Q^G(K):=K\hspace{-1pt}G \otimes_{K\hspace{-1pt}Q} K$ with $K\hspace{-1pt} N$ as a vector space. Set,\vspace{-2pt} for short, $\widehat{N}^{\bullet}= \widehat{N} \backslash \{\varepsilon\}$. The set $\big\{e_{\nu}: \nu \in \widehat{N}^{\bullet}\big\}$ is a basis of $V$ and the action of $G$ on $V$ is defined by $(nq) \hspace{1pt}\mbox{$\diamond \hspace{-4.255pt} \cdot$}\hspace{3pt} e_{\nu} = (q \triangleright \nu)(n)e_{q \triangleright \nu}$ for all $n \in N,q \in Q$.\vspace{1pt} Here, $(q \triangleright \nu)(n)=\nu(q^{-1}\cdot n)$. Consider\vspace{1pt} the character $\chi:K\hspace{-1pt} G \rightarrow K$ afforded by $V$. It is not difficult to check that $\chi$ is given by:
$$\chi\big(e^N_{\nu}q\big) =
\begin{cases}
\hspace{2pt}1\hspace{3pt} \text{ if }\hspace{2pt} \nu \neq \varepsilon \hspace{2pt} \text{ and }\hspace{2pt} q \triangleright \nu =\nu, \vspace{1pt} \\
\hspace{2pt}0\hspace{3pt} \text{ otherwise.}
\end{cases}\vspace{2pt}$$

We know from Proposition \ref{vipcharacter} that the cocharacter $c_{\tau}:= |M| e^M_{\varepsilon} \tau e^M_{\varepsilon}$ of $(K\hspace{-1pt} G)_J$ belongs to $Y$. The element $E_{\tau}:=(\chi\otimes id)(\Delta_J(c_{\tau}))$ must belong to $Y$ as well in view of Proposition \ref{character}. A large part of the rest of the proof is devoted to finding an appropriate expression for $E_{\tau}$. We start with the following computation:
\begin{align}
E_{\tau} \hspace{5pt} & \hspace{-3pt}\overset{\text{\ding{172}}}{=}\hspace{4pt} |M| \hspace{1pt} \displaystyle{\sum_{\begin{subarray}{l} \lambda_1,\lambda_2 \in \widehat{L} \\ \rho_1,\rho_2 \in \widehat{P}\end{subarray}} \big(\chi \otimes id\big)\bigg(J\Big(e^L_{\lambda_1} e^P_{\rho_1}\tau e^L_{\lambda_2} e^P_{\rho_2} \otimes e^L_{\lambda_1^{-1}}e^P_{\rho_1^{-1}}\tau e^L_{\lambda_2^{-1}}e^P_{\rho_2^{-1}}\Big)J^{-1}}\bigg) \vspace{4pt} \nonumber \\
\hspace{5pt} & \hspace{-3pt}\overset{\text{\ding{173}},\text{\ding{174}}}{=}\hspace{4pt} |M| \hspace{1pt} \displaystyle{\sum_{\begin{subarray}{l} \lambda_1,\lambda_2 \in \widehat{L} \\ \rho_1,\rho_2 \in \widehat{P}\end{subarray}} \hspace{2pt} \frac{\omega(\lambda_1,\rho_1^{-1})}{\omega(\lambda_2,\rho_2^{-1})} \hspace{2pt} \chi \big(e^L_{\lambda_1} e^P_{\rho_1}\tau e^L_{\lambda_2} e^P_{\rho_2}\big) e^L_{\lambda_1^{-1}}e^P_{\rho_1^{-1}}\tau e^L_{\lambda_2^{-1}}e^P_{\rho_2^{-1}}} \vspace{5pt} \nonumber \\
\hspace{5pt} & \hspace{-3pt}\overset{\text{\ding{175}},\text{\ding{174}}}{=}\hspace{4pt} |M| \hspace{2pt}\displaystyle{\sum_{\begin{subarray}{l} \lambda \in \widehat{L} \\ \rho \in \widehat{P}\end{subarray}} \chi\big(e^P_{\rho}\tau e^L_{\lambda}\big)e^L_{\lambda^{-1}}e^P_{\rho^{-1}}\tau e^L_{\lambda^{-1}}e^P_{\rho^{-1}}.} \label{sumetau}
\end{align}
Here, we have used:
\begin{enumerate}
\item[\ding{172}] That $\Delta(e_{\varepsilon}^M)=\sum_{\lambda \in \widehat{L}} \sum_{\rho \in \widehat{P}} \hspace{2pt} e^L_{\lambda} e^P_{\rho} \otimes e^L_{\lambda^{-1}}e^P_{\rho^{-1}}$. \vspace{1pt}
\item[\ding{173}] Definition of $J$. \vspace{1pt}
\item[\ding{174}] That $\{e_{\lambda}^L\}_{\lambda \in \widehat{L}}$ and $\{e_{\rho}^P\}_{\rho \in \widehat{P}}$ are complete sets of orthogonal idempotents in $K\hspace{-0.9pt}L$ and $K\hspace{-0.9pt}P$, respectively. \vspace{1pt}
\item[\ding{175}] That $\chi$ is a character: $\chi(gh)=\chi(hg)$ for all $g,h \in G$. And that $L$ and $P$ commute with one another. \vspace{1pt}
\end{enumerate}

We\vspace{-2pt} next calculate the scalar $\chi\big(e^P_{\rho}\tau e^L_{\lambda}\big)$ occurring in the sum \eqref{sumetau}. We use that $e^L_{\lambda} = \sum_{\begin{subarray}{l} \nu \in \widehat{N} \\ \nu \vert_L = \lambda \end{subarray}} e^N_{\nu}$. We have:
$$\chi\big(e^P_{\rho}\tau e^L_{\lambda}\big) \hspace{1pt} = \hspace{2pt} \frac{1}{|P|}\sum_{\sigma \in P} \sum_{\begin{subarray}{c} \nu \in \widehat{N} \vspace{0.75pt} \\ \nu \vert_L = \lambda \end{subarray}} \rho(\sigma^{-1})\chi\big(e^N_{\nu}\sigma \tau\big).$$
By hypothesis (v), $\sigma \tau$ does not fix non-trivial characters in $\widehat{N}$ when $\sigma \neq 1$. Then, the above equality takes the form:
\begin{equation}\label{Eq1}
\hspace{-4mm}\chi\big(e^P_{\rho}\tau e^L_{\lambda}\big) \hspace{1pt} = \hspace{2pt} \frac{1}{|P|}\sum_{\begin{subarray}{c} \nu \in \widehat{N} \vspace{0.75pt}\\ \nu \vert_L = \lambda \end{subarray}} \chi(e^N_{\nu}\tau) \hspace{1pt} = \hspace{2pt} \frac{1}{|P|} \hspace{1.5pt} \# \big\{\nu \in \widehat{N}^{\bullet} \hspace{1pt}:\hspace{1pt} \nu\vert_L=\lambda \hspace{2pt}\text{ and }\hspace{2pt} \tau \triangleright \nu =\nu\big\}.
\end{equation}

We next find out the value of the right-hand side term in this equation. Put $_{\tau}N=\{\tau \cdot n-n:n\in N\}$. The condition $\tau \triangleright \nu=\nu$ amounts to  $\nu \vert_{_{\tau}\hspace{-1pt}N}=\varepsilon$. Recall that our hypothesis (iii) is that $N=L\oplus (\tau \cdot L)$. This implies that $N=L+{}_{\tau}N$. Observe that the conditions $\tau \triangleright \nu=\nu$ and $\nu\vert_L=\lambda$ define $\nu$ uniquely. And $\nu$ satisfies such conditions if and only if $\lambda\vert_{L\cap\hspace{1pt} {}_{\tau}\hspace{-1pt}N} = \varepsilon$. Equation \ref{Eq1} now reads as:
$$\chi\big(e^P_{\rho}\tau e^L_{\lambda}\big) =
\begin{cases}
\hspace{2pt}\frac{1}{|P|}\hspace{3pt} \text{ if }\hspace{2pt} \lambda\vert_{L\cap\hspace{1pt} {}_{\tau}\hspace{-1pt}N} = \varepsilon \hspace{2pt}\text{ and } \hspace{2pt} \lambda\neq\varepsilon, \vspace{1pt} \\
\hspace{6pt}0\hspace{6pt} \text{ otherwise.}
\end{cases}$$

We return to the sum \eqref{sumetau}. We make the changes of variables $\rho \mapsto \rho^{-1}$ and $\lambda \mapsto \lambda^{-1}$ and substitute the value of $\chi\big(e^P_{\rho}\tau e^L_{\lambda}\big)$. As before, we set $\widehat{L}^{\bullet}= \widehat{L} \backslash \{\varepsilon\}$. We get:
$$E_{\tau} \hspace{1pt} = \hspace{1pt} \frac{|M|}{|P|}\hspace{1pt} \sum_{\rho\in \widehat{P}} \hspace{1pt}\sum_{\begin{subarray}{c} \lambda\in \widehat{L}^{\bullet} \vspace{0.3pt} \\ \lambda\vert_{L \cap \hspace{1pt} {}_{\tau}\hspace{-1pt}N}=\varepsilon \end{subarray}} e^L_{\lambda}e^P_{\rho} \tau e^L_{\lambda}e^P_{\rho}.$$
We simplify further this expression by using the equality:
$$\sum_{\rho\in \widehat{P}} e^P_{\rho}\otimes e^P_{\rho} \hspace{1pt} = \hspace{1pt} \frac{1}{|P|}\sum_{\sigma \in P} \sigma \otimes \sigma^{-1}.$$
We obtain:
\begin{equation}\label{Eq2}
E_{\tau} \hspace{1pt} = \hspace{1pt} \sum_{\sigma \in P} \sigma \Bigg(\sum_{\begin{subarray}{c} \lambda\in \widehat{L}^{\bullet} \vspace{0.3pt} \\ \lambda\vert_{L \cap \hspace{1pt} {}_{\tau}\hspace{-1pt}N} =\varepsilon \end{subarray}} e^L_{\lambda} \tau e^L_{\lambda}\hspace{2pt}\Bigg)\sigma^{-1}.
\end{equation}

The next step is to simplify the sum in the brackets. Since $N=L\oplus (\tau \cdot L)$, we have an isomorphism $L\oplus L\simeq N$ given by $(l_1,l_2)\mapsto l_1+\tau\cdot l_2$. Under this identification, we can write the action of $\tau$ on $N\simeq L\oplus L$ in matrix form as:
$$\tau=
\begin{pmatrix}
0 & \alpha \\
id & \gamma
\end{pmatrix}.$$
Here, $\alpha,\gamma:L \rightarrow L$ are the following group homomorphisms: $\alpha=\pi_L \circ \tau^2$ and $\gamma=\tau^{-1} \circ \pi_{\tau\cdot L} \circ \tau^2$, where $\pi_L:N \rightarrow L$ and $\pi_{\tau\cdot L}:N \rightarrow \tau\cdot L$ are the projections attached to the direct sum $N=L\oplus (\tau\cdot L)$. \par \vspace{2pt}

We will identify $\widehat{N}$ with $\widehat{L}\oplus \widehat{L}$ as well.
Bear in mind that $e^L_{\lambda} = \sum_{\lambda' \in \widehat{L}} e^N_{(\lambda,\lambda')}$. We compute:
$$\begin{array}{rl}
e^L_{\lambda}\tau e^L_{\lambda} & \hspace{-2.5pt} = \hspace{3pt} {\displaystyle \sum_{\lambda_1,\lambda_2\in \widehat{L}} e^N_{(\lambda,\lambda_1)}\tau e^N_{(\lambda,\lambda_2)}} \vspace{4pt} \\
 & \hspace{-2.5pt} = \hspace{3pt} {\displaystyle \sum_{\lambda_1,\lambda_2\in \widehat{L}} \tau e^N_{\tau^{-1} \triangleright (\lambda,\lambda_1)} e^N_{(\lambda,\lambda_2)}} \vspace{4pt} \\
 & \hspace{-2.5pt} = \hspace{3pt} {\displaystyle \sum_{\lambda_1,\lambda_2 \in \widehat{L}} \tau e^N_{(\lambda_1,\lambda\circ \alpha + \lambda_1\circ \gamma)} e^N_{(\lambda,\lambda_2)}} \vspace{4pt} \\
 & \hspace{-2.5pt} = \hspace{3pt} \tau e^N_{(\lambda,\lambda \circ (\alpha+\gamma))}
\end{array}$$
We replace this in Equation \ref{Eq2}. We get:
$$E_{\tau} \hspace{1pt} = \hspace{1pt} \sum_{\sigma \in P}  \sum_{\begin{subarray}{c} \lambda\in \widehat{L}^{\bullet} \vspace{0.3pt} \\ \lambda\vert_{L \cap\hspace{1pt}  {}_{\tau}\hspace{-1pt}N} =\varepsilon \end{subarray}} \sigma \tau e^N_{(\lambda,\lambda \circ (\alpha+\gamma))} \sigma^{-1}.$$

We next describe $L\cap {}_{\tau}N$. For $(l_1,l_2) \in L \oplus L$, we have:
$$(\tau-id) \begin{pmatrix} l_1 \\ l_2\end{pmatrix} = \begin{pmatrix} \alpha(l_2)-l_1 \\ l_1+(\gamma-id)(l_2) \end{pmatrix}.$$
This element belongs to $L$ if and only if $l_1=(id-\gamma)(l_2)$. Then:
$$(\tau-id) \begin{pmatrix} l_1 \\ l_2 \end{pmatrix} = \begin{pmatrix}
(\alpha+\gamma-id)(l_2) \\
0
\end{pmatrix}.$$
We thus see that $\lambda \in \widehat{L}$ is trivial on $L\cap {}_{\tau} N$ if and only if $\lambda \circ (\alpha+\gamma-id)=\varepsilon$; or, equivalently, $\lambda \circ (\alpha+\gamma)=\lambda$. In this case, the\vspace{1pt} character $(\lambda,\lambda)$ in $\widehat{N}$ is $\tau$-invariant. We show\vspace{1pt} that the set of characters $\lambda$ satisfying $\lambda \circ (\alpha+\gamma-id)=\varepsilon$ is not empty. Our hypothesis (iv) states that $N^{\tau}\neq \{1\}$. Then, there is a non-trivial $(l_1,l_2)$ such that:
$$\begin{pmatrix}
0 & \alpha \\
id & \gamma
\end{pmatrix}
\begin{pmatrix}
l_1 \\
l_2
\end{pmatrix} =
\begin{pmatrix}
l_1 \\
l_2
\end{pmatrix}.$$
This means that $\alpha(l_2) = l_1$ and $(\alpha+\gamma-id)(l_2)=0$. The non-triviality of $(l_1,l_2)$ implies that $\alpha+\gamma-id:L \to L$ is not invertible.  Therefore, there is $\lambda \in \widehat{L}$ such that $\lambda \neq \varepsilon$ and $\lambda \circ (\alpha+\gamma-id)=\varepsilon$. \par \vspace{2pt}

Summing this up, we finally arrive at the desired expression for $E_{\tau}$:
$$E_{\tau} \hspace{1pt} = \hspace{1pt} \sum_{\sigma \in P} \sum_{\begin{subarray}{c} \lambda \in \widehat{L}^{\bullet} \vspace{0.3pt} \\ \lambda \circ (\alpha+\gamma)=\lambda \end{subarray}}
\sigma \tau e^N_{(\lambda,\lambda)} \sigma^{-1}.$$

On the other hand, proceeding in a similar fashion with the cocharacter $c_{\tau^{-1}}:=|M|e^M_{\varepsilon}\tau^{-1}e^M_{\varepsilon}$ we obtain the following element:
$$E_{\tau^{-1}} \hspace{1pt} = \hspace{1pt} \sum_{\sigma \in P} \sum_{\begin{subarray}{c} \lambda \in \widehat{L}^{\bullet} \vspace{0.3pt} \\ \lambda \circ (\alpha+\gamma)=\lambda \end{subarray}}
\sigma \tau^{-1} e^N_{(\lambda,\lambda)} \sigma^{-1}.$$

Both elements, $E_{\tau}$ and $E_{\tau^{-1}}$, belong to $Y$. The product $E_{\tau}E_{\tau^{-1}}$ belongs to $Y$ as well. The next step is to calculate $E_{\tau}E_{\tau^{-1}}$:
$$E_{\tau}E_{\tau^{-1}} \hspace{1pt} = \hspace{1pt} \sum_{\sigma_1,\sigma_2 \in P} \hspace{3pt} \sum_{\begin{subarray}{c} \lambda_1 \in \widehat{L}^{\bullet} \vspace{0.3pt} \\ \lambda_1 \circ (\alpha+\gamma)=\lambda_1 \end{subarray}} \hspace{3pt} \sum_{\begin{subarray}{c} \lambda_2 \in \widehat{L}^{\bullet} \vspace{0.3pt} \\ \lambda_2 \circ (\alpha+\gamma)=\lambda_2 \end{subarray}} \sigma_1 \tau e^N_{(\lambda_1,\lambda_1)}\sigma_1^{-1}\sigma_2 \tau^{-1} e^N_{(\lambda_2,\lambda_2)}\sigma_2^{-1}.$$
We use that the character $(\lambda_2,\lambda_2)$ is $\tau$-invariant. We have:
$$E_{\tau}E_{\tau^{-1}} \hspace{1pt} = \hspace{1pt} \sum_{\sigma_1,\sigma_2 \in P} \hspace{3pt} \sum_{\begin{subarray}{c} \lambda_1 \in \widehat{L}^{\bullet} \vspace{0.3pt} \\ \lambda_1 \circ (\alpha+\gamma)=\lambda_1 \end{subarray}} \hspace{3pt} \sum_{\begin{subarray}{c} \lambda_2 \in \widehat{L}^{\bullet} \vspace{0.3pt}\\ \lambda_2 \circ (\alpha+\gamma)=\lambda_2 \end{subarray}} \sigma_1 \tau e^N_{(\lambda_1,\lambda_1)} e^N_{(\sigma_1^{-1}\sigma_2) \hspace{0.5pt} \triangleright \hspace{0.5pt} (\lambda_2,\lambda_2)}\sigma_1^{-1}\sigma_2\tau^{-1}\sigma_2^{-1}.$$
The product\vspace{1.5pt} of the idempotents is non-zero if and only if $(\sigma_1^{-1}\sigma_2) \triangleright (\lambda_2,\lambda_2)=(\lambda_1,\lambda_1)$. This implies that $\big((\sigma_1^{-1}\sigma_2)^{-1} \tau (\sigma_1^{-1}\sigma_2)\big) \triangleright (\lambda_2,\lambda_2)=(\lambda_2,\lambda_2)$. Our hypothesis (v) states\vspace{0.5pt} that $\big(\widehat{N}^{\tau}\big) \cap \big(\widehat{N}^{\sigma \tau \sigma^{-1}}\big)=\{\varepsilon\}$ whenever $\sigma \neq 1$. Then,\vspace{1pt} the only non-zero contributions\vspace{1.5pt} to the previous sum occur when $\sigma_1=\sigma_2$. And,\vspace{1pt} in this case, $\lambda_1=\lambda_2$. We obtain:
$$E_{\tau}E_{\tau^{-1}} \hspace{2pt} = \hspace{4pt} \sum_{\sigma \in P} \hspace{3pt} \sum_{\begin{subarray}{c} \lambda \in \widehat{L}^{\bullet} \vspace{0.3pt} \\ \lambda \circ (\alpha+\gamma)=\lambda \end{subarray}} \sigma e^N_{(\lambda,\lambda)} \sigma^{-1} \hspace{2pt} = \hspace{4pt} \sum_{\sigma \in P} \hspace{3pt} \sum_{\begin{subarray}{c} \lambda \in \widehat{L}^{\bullet} \vspace{0.3pt} \\ \lambda \circ (\alpha+\gamma)=\lambda \end{subarray}}  e^N_{\sigma \hspace{0.5pt} \triangleright \hspace{0.5pt} (\lambda,\lambda)}.$$

Recall that $L$ and $P$ commute by hypothesis (i). For every $l \in L$, we have:
$$\big(\sigma \triangleright (\lambda,\lambda)\big)(l) = (\lambda,\lambda)(\sigma^{-1}l\sigma) = (\lambda,\lambda)(l) = \lambda(l).$$
Then, for every $\sigma \in P$, there is $\mu(\sigma) \in \widehat{L}$ such that $\sigma \triangleright (\lambda,\lambda)=(\lambda,\mu(\sigma))$. Since $(\lambda,\lambda)$ is $\tau$-invariant and non-trivial, it cannot be $\sigma$-invariant for $\sigma \in P$ with $\sigma \neq 1$. Otherwise, $(\lambda,\lambda)$ would be $\sigma \tau$-invariant, contradicting\vspace{-1pt} our hypothesis (v). Hence, $\mu(\sigma_1) \neq \mu(\sigma_2)$ when $\sigma_1\neq \sigma_2$. This yields\vspace{0.5pt} that $\mu(\sigma)$ runs one-to-one over all $\widehat{L}$ when $\sigma$ runs in $P$. Now we get:
$$E_{\tau}E_{\tau^{-1}} \hspace{3pt} = \hspace{2pt} \sum_{\begin{subarray}{c} \lambda \in \widehat{L}^{\bullet} \vspace{0.3pt} \\ \lambda \circ (\alpha+\gamma)=\lambda \end{subarray}} \hspace{3pt} \sum_{\sigma \in P} e^N_{\sigma \hspace{0.5pt} \triangleright \hspace{0.5pt} (\lambda,\lambda)} \hspace{4pt} = \hspace{2pt} \sum_{\begin{subarray}{c} \lambda \in \widehat{L}^{\bullet} \vspace{0.3pt} \\ \lambda \circ (\alpha+\gamma)=\lambda \end{subarray}} \hspace{4pt} \sum_{\lambda' \in \widehat{L}} e^N_{(\lambda,\lambda')} \hspace{4pt} = \hspace{2pt} \sum_{\begin{subarray}{c} \lambda \in \widehat{L}^{\bullet} \vspace{0.3pt} \\ \lambda \circ (\alpha+\gamma)=\lambda \end{subarray}} e_{\lambda}^L.$$

In particular, this gives that $E_{\tau}E_{\tau^{-1}} \in K\hspace{-1pt}L$. Then, $\Delta_J(E_{\tau}E_{\tau^{-1}})=\Delta(E_{\tau}E_{\tau^{-1}})$. Let $\varphi \in \widehat{L}$ be such that $\varphi\circ (\alpha+\gamma)=\varphi$. The next step\vspace{1.5pt} is to compute the element $(\varphi \otimes id)(\Delta_J(E_{\tau}E_{\tau^{-1}}))$. We have:
$$\begin{array}{ll}
(\varphi \otimes id)(\Delta_J(E_{\tau}E_{\tau^{-1}})) & =\hspace{5pt} {\displaystyle \sum_{\begin{subarray}{c} \lambda \in \widehat{L}^{\bullet} \vspace{1pt} \\ \lambda \circ (\alpha+\gamma)=\lambda \end{subarray}} \hspace{3pt} \sum_{\phi \in \widehat{L}} \varphi(e_{\phi}^L) e_{\phi^{-1}\lambda}^L} \vspace{5pt} \\
 & =\hspace{5pt} {\displaystyle \sum_{\begin{subarray}{c} \lambda \in \widehat{L}^{\bullet} \vspace{1pt} \\ \lambda \circ (\alpha+\gamma)=\lambda \end{subarray}} e_{\varphi^{-1}\lambda}^L} \vspace{5pt} \\
 & =\hspace{5pt} {\displaystyle \sum_{\begin{subarray}{c} \lambda \in \widehat{L}\backslash\{\varphi^{-1}\hspace{-0.5pt}\} \vspace{1pt}\\ \lambda \circ (\alpha+\gamma)=\lambda \end{subarray}} e_{\lambda}^L.}
\end{array}$$
In the last equality, we used that the set $\{\lambda \in \widehat{L} : \lambda \circ (\alpha+\gamma)=\lambda\}$ is a subgroup of $\widehat{L}$. By Proposition \ref{sub}, $Y \cap (K\hspace{-1pt}L)$ is a Hopf order of $K\hspace{-1pt}L$ over $R$.\vspace{1.25pt} The element $(\varphi \otimes id)(\Delta_J(E_{\tau}E_{\tau^{-1}}))$ belongs to $Y \cap (K\hspace{-1pt}L)$ in light of Proposition \ref{character}.\vspace{1.25pt} We exploit further the previous equality with the following calculation: \par \vspace{2pt}
$$\prod_{\begin{subarray}{c} \varphi \in \widehat{L}^{\bullet} \vspace{1pt} \\ \varphi \circ (\alpha+\gamma)=\varphi \end{subarray}}(\varphi \otimes id)(\Delta_J(E_{\tau}E_{\tau^{-1}})) \hspace{1pt} = \hspace{1pt} e_{\varepsilon}^L.$$
This yields that $e_{\varepsilon}^L \in Y \cap (K\hspace{-1pt}L)$. Thus, $e_{\varepsilon}^L \in Y$. Propositions \ref{cond:unique} and \ref{squeeze} finally apply to obtain $Y=X$.
\end{proof}

The following consequence of the uniqueness property is noteworthy:

\begin{remark}
The Hopf order $X$ constructed in the above proof is the unique Hopf order of $K\hspace{-1pt} G$ such that $J \in X \otimes_R X$.
\end{remark}

The number of orders in a semisimple Hopf algebra over a number field is finite by \cite[Theorem 1.8]{M}. Theorem \ref{thm:main2} brings to light a different behavior of the number of Hopf orders in twisted group algebras in comparison to that in group algebras on abelian groups. \par \vspace{2pt}

For example, the number of Hopf orders of the group algebra on $C_p,$ with $p$ prime, tends to infinity when the base number field is enlarged in a suitable way, see \cite[Theorem 3 and p. 21]{TO} or the compilation in \cite[Section 3]{CM2}. However, for twisted group algebras, Theorem \ref{thm:main2} shows that such a number can be constantly one. This phenomenon already appeared in our study of the Hopf orders of Nikshych's Hopf algebra, see \cite[Theorem 6.15 and Remark 6.13]{CM2}.

\section{Examples}\label{example}

We illustrate Theorem \ref{thm:main2} with several examples:

\subsection{The basic example}
Let $\F_q$ be the finite field with $q$ elements. For $n \in \Na$, consider the field extension $\F_{q^n}/\F_q$. By fixing a basis of $\F_{q^n}$ as a $\F_q$-vector space, we get a linear isomorphism $\F_{q^n} \simeq \F_q^n$. The product of $\F_{q^n}$ induces an injective algebra homomorphism
$\Phi:\F_{q^n}\to \M_n(\F_q)$. \par \vspace{2pt}

Consider the groups $Q=\SL_{2n}(q)$ and $N=\F_q^{2n}$, where $Q$ acts on $N$ in the natural way. Write $B=\{v_1,\ldots, v_{2n}\}$ for the standard basis of $\F_q^{2n}$ as a $\F_q$-vector space. Let $L$ and $L'$ be\vspace{1.25pt} the subspaces of $\F_q^{2n}$ spanned by $\{v_1,\ldots, v_n\}$ and $\{v_{n+1},\ldots, v_{2n}\}$ respectively. Let $P$ be the following subgroup of $Q$:
$$P=\bigg\{
\begin{pmatrix}
\textrm{Id} & \hspace{-3pt}\Phi(a) \\
0  & \hspace{-3pt}\textrm{Id}
\end{pmatrix} : a \in \F_{q^n}\bigg\}.$$
Finally, we pick in $Q$ the block matrix
$$\tau=\begin{pmatrix}
\textrm{Id} & \hspace{-3pt}0 \\
\textrm{Id} & \hspace{-3pt}\textrm{Id}
\end{pmatrix}.$$
We already have all the prescribed data to apply Theorem \ref{thm:main2}. It is easy to see that $N,Q,L,P,$ and $\tau$ satisfy the hypotheses (i)-(iv). Notice that $N^{\tau}=L'$. \par \vspace{2pt}

We next justify that the hypothesis (v) is also satisfied. First, we check that $\widehat{N}^{\sigma\tau} = \{\varepsilon\}$ for every $\sigma \in P$ with $\sigma \neq 1$. We write $\sigma$ as
$$\sigma =
\begin{pmatrix}
\textrm{Id} & \hspace{-3pt}\Phi(a) \\
0 & \hspace{-3pt}\textrm{Id}
\end{pmatrix},$$
with $a \neq 0$. Then, the following determinant is nonzero:
$$\det(\sigma\tau -1) =
\det \begin{pmatrix}
\Phi(a) & \hspace{-3pt}\Phi(a) \\
    \textrm{Id}  & \hspace{-3pt}0
\end{pmatrix} = \det (\Phi(-a)).$$
The matrix $\sigma\tau-1$ acts on $N$ and this action is invertible. Hence, the action of $\sigma\tau-1$ on $\widehat{N}$ is invertible as well. This yields that $\widehat{N}^{\sigma\tau} = \{\varepsilon\}$. \par \vspace{2pt}

Second, we check that $\widehat{N}^{\tau} \cap \widehat{N}^{\sigma \tau \sigma^{-1}}=\{\varepsilon\}$ for $\sigma$ as before. We can write every element of $\widehat{N}$ as a pair $(\varphi,\psi)$, with $\varphi \in \widehat{L}$ and $\psi \in \widehat{L'}.$ One can easily\vspace{1pt} show that $\widehat{N}^{\tau} = \big\{(\varphi,\varepsilon): \varphi \in \widehat{L}\big\}$ and $\widehat{N}^{\sigma\tau \sigma^{-1}} = \big\{(\varphi, \Phi(-a)^{-1} \triangleright \varphi) : \varphi \in \widehat{L}\big\}$. Since $\Phi(-a)$ is invertible, we obtain that $\widehat{N}^{\tau}\cap \widehat{N}^{\sigma \tau \sigma^{-1}}=\{\varepsilon\}$. \par \vspace{2pt}

We have verified that all the hypotheses of Theorem \ref{thm:main2} hold. We denote the resulting twisted Hopf algebra $\big(K(\F_q^{2n} \rtimes \SL_{2n}(q))\big)_J$ by $H_{q,n}$. The requirement that $K$ is large enough is satisfied in this example if $K$ contains a primitive $p$th root of unity, where $p=\textrm{char}(\F_q)$.
\par \vspace{2pt}

In this way, we constructed an infinite family of semisimple Hopf algebras, parameterized by $q$ and $n$, which admit a unique Hopf order over $R$.

\subsection{Variations of the basic example} \enlargethispage{1.5\baselineskip}
Observe that the previous construction similarly works for $Q=\GL_{2n}(q)$. \par \vspace{2pt}

We next show that it works for $Q=\Sp_{2n}(q)$ as well. Recall that
$$\Sp_{2n}(q)=\bigg\{A \in \GL_{2n}(q): A^t\begin{pmatrix}
0 & \hspace{-3pt}\textrm{Id} \\
-\textrm{Id} & \hspace{-3pt}\textrm{0}
\end{pmatrix} A= \begin{pmatrix}
0 & \hspace{-3pt}\textrm{Id} \\
-\textrm{Id} & \hspace{-3pt}\textrm{0}
\end{pmatrix}\bigg\}.$$
Consider\vspace{0.25pt} the bilinear form $\mathcal{T}:\F_{q^n} \times \F_{q^n} \rightarrow \F_q$ defined by the trace of the field extension $\F_{q^n}/\F_q$. We know\vspace{0.5pt} that this form is non-degenerate and symmetric. (When $q$ is even, we also know that the map $\F_{q^n}\rightarrow \F_q, x \mapsto \textrm{Tr}_{\F_{q^n}/\F_q}(x^2)=\textrm{Tr}_{\F_{q^n}/\F_q}(x)^2$, is nonzero.) A result from\vspace{0.25pt} the theory of bilinear forms ensures the existence of an orthogonal\vspace{0.12pt} basis for $\F_{q^n}$ as a $\F_q$-vector space. Expressing the product of $\F_{q^n}$ with respect to this basis,\vspace{0.12pt} we get an algebra homomorphism $\Psi:\F_{q^n}\to \M_n(\F_q)$ with the additional property that $\Psi(a)$ is symmetric for all $a \in \F_q$. Now, set
$$P=\Bigg\{
\begin{pmatrix}
\textrm{Id} & \hspace{-3pt}\Psi(a) \\
0  & \hspace{-3pt}\textrm{Id}
\end{pmatrix} : a \in \F_{q^n}\Bigg\}.$$
One can easily check that $P \subset Q$ and $\tau \in Q$. The hypotheses (i)-(iv) of Theorem \ref{thm:main2} are likewise satisfied. For\vspace{-1pt} the hypothesis (v), one can argue as before to check the required conditions on the invariant subsets of $\widehat{N}$. Or,\vspace{1pt} alternatively, use that $\Psi(a)=D\Phi(a)D^{-1}$ for all $a \in \F_{q^n}$, where $D$ is an invertible matrix.

\subsection{Composition of basic examples}
Let $n_1,n_2,\ldots,n_k \in \Na$ be such that $n=\sum_i n_i$. For a fixed $q$, consider the family of Hopf algebras $H_{q,n_i}$ and the tensor product Hopf algebra
$$\bigotimes_i H_{q,n_i} \simeq K\Big(\prod_i \big(\F_q^{2n_i} \rtimes \SL_{2n_i}(q)\big)\Big)_{\otimes J_i} \simeq K\Big(\F_q^{2n} \rtimes \Big(\prod_i \SL_{2n_i}(q)\Big)\Big)_{\Upsilon},$$
where we write $\Upsilon=\otimes J_i$. Let $X_{q,n_i}$ denote the unique Hopf order of $H_{q,n_i}$. Then, $Y:=\otimes_i \hspace{1pt} X_{q,n_i}$ is a Hopf order of $K\hspace{-1.25pt}\big(\F_q^{2n} \rtimes (\prod_i \SL_{2n_i}(q))\big)_{\Upsilon}$. Since $Y$ contains all the primitive idempotents of $K\F_q^{2n}$, we have that $\Upsilon \in Y\otimes_R Y$ and $Y$ is unique by Proposition \ref{cond:unique}. \par \vspace{2pt}

The diagonal embedding $\prod_i \SL_{2n_i}(q) \hookrightarrow \SL_{2n}(q)$ induces an embedding of Hopf algebras
$$\bigotimes_i H_{q,n_i} \hookrightarrow K\Big(\F_q^{2n} \rtimes \SL_{2n}(q)\Big)_{\Upsilon}.$$
By intersecting with $\otimes_i \hspace{1pt} H_{q,n_i}$, we see that any Hopf order $X$ of $K\hspace{-1.25pt}\big(\F_q^{2n} \rtimes \SL_{2n}(q)\big)_{\Upsilon}$ contains all the primitive idempotents of $K\F_q^{2n}$ and satisfies $\Upsilon \in X\otimes_R X$. Then, $X$ is also a Hopf order of the group algebra $K\hspace{-1.25pt}\big(\F_q^{2n} \rtimes \SL_{2n}(q)\big)$. Theorem \ref{thm:main1} and Proposition \ref{cond:unique} show that $K\hspace{-1.25pt}\big(\F_q^{2n} \rtimes \SL_{2n}(q)\big)_{\Upsilon}$ has a unique Hopf order.

\section{An application}\label{app}

In \cite[Section 5]{CM3}, we posed the following question:

\begin{question}
Let $G$ be a finite non-abelian simple group. Let $\Omega$ be a non-trivial twist for $\Co G$ arising from a $2$-cocycle on an abelian subgroup of $G$. Can $(\Co G)_{\Omega}$ admit a Hopf order over a number ring?
\end{question}

The results obtained so far in \cite{CCM} and \cite[Theorem 3.3]{CM3} support a negative answer. In this final section, we provide one more instance of a partial negative answer through $\PSL_{2n+1}(q)$. The strategy of proof deployed here differs from that in \cite{CCM} and \cite[Theorem 3.3]{CM3}. Here, we embed  $H_{q,n}$ in a twist of $K \PSL_{2n+1}(q)$ and exploit the concrete form of the unique Hopf order of $H_{q,n}$. Compared with \cite[Theorem 6.3]{CCM}, we note that the following proof is constructive and does not rely either on the classification of the finite simple groups or on that of the minimal simple groups.

\begin{theorem}\label{PSL2n1}
Let $K$ be a number field with ring of integers $R$. Let $p$ be a prime number and $q=p^m$ with $m \geq 1$. Assume that $K$ contains a primitive $p$th root of unity $\zeta$. There exists a twist $J$ for the group algebra $K \PSL_{2n+1}(q)$, arising from a $2$-cocycle on an abelian subgroup, such that $(K \PSL_{2n+1}(q))_J$ does not admit a Hopf order over $R$.
\end{theorem}

\begin{proof}
Let $\pi:\SL_{2n+1}(q) \rightarrow \PSL_{2n+1}(q)$ denote\vspace{0.75pt} the canonical projection. As in Section \ref{example}, write $N=\F_q^{2n}$. We have an embedding $\iota:N \rtimes \SL_{2n}(q) \rightarrow \PSL_{2n+1}(q)$ given by:
$$(v,A) \mapsto \pi \hspace{-2pt}\begin{pmatrix} A & v \\ 0 & 1 \end{pmatrix}.$$
We identify $N \rtimes \SL_{2n}(q)$ with its image through $\iota$ and view it as a subgroup of $\PSL_{2n+1}(q)$. The twist $J$ for $K(N \rtimes \SL_{2n}(q))$ used in the construction of $H_{q,n}$ allows us to twist $K\PSL_{2n+1}(q)$ as well. Thus, we can consider $H_{q,n}$ as a Hopf subalgebra of $(K\PSL_{2n+1}(q))_J$. \par \vspace{2pt}

We will next prove that $(K\PSL_{2n+1}(q))_J$ does not admit a Hopf order over $R$ by contradiction. Suppose that $(K\PSL_{2n+1}(q))_J$ admits a Hopf order $Y$ over $R$. Then, $Y \cap H_{q,n}$ is a Hopf order of $H_{q,n}$ over $R$. We saw in Section \ref{example} that $H_{q,n}$ admits a unique Hopf order $X$ over $R$. So,  $X=Y \cap H_{q,n}$. And we know that $J \in X \otimes_R X$. This implies that $J \in Y \otimes_R Y$ and that $Y$ is a Hopf order of $K\PSL_{2n+1}(q)$. \par \vspace{2pt}

We write $u_{\pi(B)}$ for the element $\pi(B) \in \PSL_{2n+1}(q)$ when viewed in the group algebra $K\PSL_{2n+1}(q)$. For $r \in \F_q$, let $x_{ij}(r)$ denote the elementary matrix in $\SL_{2n+1}(q)$ with $r$ in the $(i,j)$-entry. Recall that $X$ (and consequently $Y$) contains the primitive idempotents of $K\hspace{-1pt} N$. Then, for a character $\chi:\F_q \to K^{\times}$ of $(\F_q,+)$, the element
$$\frac{1}{q}\sum_{r \in \F_q} \chi(r) u_{\pi(x_{1\hspace{0.33pt}2n+1}(r))}$$
belongs to $Y$. Since $\PSL_{2n+1}(q) \subset Y$, we can multiply that element by elements of $\PSL_{2n+1}(q)$ and produce other elements in $Y$. In particular, $Y$ contains the element
$$\frac{1}{q}\sum_{r \in \F_q} \chi(r) u_{\pi(x_{ij}(r))}.$$
Consider now $(\F_p,+)$ as a subgroup of $(\F_q,+)$. Summing over all characters $\chi$ such that $\chi \vert_{\F_p}$ is trivial, we get that
$$\frac{1}{p}\sum_{r \in \F_p} u_{\pi(x_{ij}(r))}$$
belongs to $Y$. \par \vspace{2pt}

Put $g_1=\pi(x_{12}(1)), g_2=\pi(x_{13}(1)),$ and $g_3=\pi(x_{23}(1))$. The subgroup generated by $g_1,g_2,$ and $g_3$ is a Heisenberg group of order $p^3$. The previous arguments show that $Y$ contains the element
$$\frac{1}{p^2}\sum_{r,s=0}^{p-1} u_{g_1^{r}g_3^{s}} = \Big(\frac{1}{p}\sum_{r \in \F_p} u_{\pi(x_{12}(r))}\Big)\Big(\frac{1}{p}\sum_{s \in \F_p} u_{\pi(x_{23}(s))}\Big).$$

Consider the $p$-th dimensional irreducible representation of the Heisenberg group given by:
$$\begin{array}{l}
g_1 \mapsto  E_{12} + E_{23} + \ldots + E_{p-1 p} + E_{p1}, \vspace{3pt} \\
g_2 \mapsto \textrm{diag}(\zeta,\zeta,\ldots,\zeta), \vspace{3pt} \\
g_3 \mapsto \textrm{diag}(1,\zeta,\ldots, \zeta^{p-1}).
\end{array}$$
Here, $E_{ij}$ denotes\vspace{1.5pt} the matrix in $\M_p(K)$ with $1$ in the $(i,j)$-entry and zero elsewhere. A direct calculation reveals that $\frac{1}{p^2}\sum_{r,s} u_{g_1^{r}g_3^{s}}$ maps to $\frac{1}{p}\sum_{i=1}^{p} E_{i1}$. Since $\frac{1}{p^2}\sum_{r,s} u_{g_1^{r}g_3^{s}}$ belongs to $Y$, it satisfies\vspace{1.5pt} a monic polynomial with coefficients in $R$. However, $\frac{1}{p}\sum_{i=1}^{p} E_{i1}$ does not, reaching so a contradiction.
\end{proof}

The statement for the complexified group algebra is established as in the proof of \cite[Corollary 2.4]{CM1}:

\begin{corollary}
There exists a twist $J$ for the group algebra $\Co \PSL_{2n+1}(q)$, arising from a $2$-cocycle on an abelian subgroup, such that the complex semisimple Hopf algebra $(\Co \PSL_{2n+1}(q))_J$ does not admit a Hopf order over any number ring.
\end{corollary}

\end{document}